\documentclass[10pt]{article}
\usepackage[utf8]{inputenc}
\usepackage[english]{babel}
\usepackage{amssymb,amsmath,amsthm}
\usepackage{indentfirst}
\usepackage{enumitem}
\usepackage{tocloft}
\usepackage{MnSymbol}
\usepackage{mathtools}

\theoremstyle{definition}

\newtheorem{theorem}{Theorem}[section]

\newtheorem{proposition}[theorem]{Proposition}
\newtheorem{lemma}[theorem]{Lemma}
\newtheorem{corollary}[theorem]{Corollary}

\newtheorem{observation}[theorem]{Observation}

\newtheorem{notation}[theorem]{Notation}
\newtheorem{definition}[theorem]{Definition}

\title{Polynomial identities with involution for the algebra of $3 \times 3$ upper triangular matrices}
\author{Dimas Jos\'e Gon\c{c}alves
\\
Universidade Federal de S\~ao Carlos \\
Departamento de Matem\'{a}tica \\
13565-905 S\~ao Carlos, SP, Brasil\\
e-mail: \texttt{dimasjg@ufscar.br} \\
\\
Dalton Couto Silva \\
Instituto Federal de Educação, Ciência e Tecnologia de São Paulo \\
11665-071 Caraguatatuba, SP, Brasil\\
e-mail: \texttt{dalton.couto@ifsp.edu.br}
}

\begin{document}

\maketitle



\

\begin{abstract}
Let $\mathbb{F}$ be a field of characteristic $p$, and let
$UT_n(\mathbb{F})$ be the algebra   
of $n \times n$ upper triangular matrices over $\mathbb{F}$ with an involution of the first kind. 
In this paper we describe:
the set of all 
$*$-central polynomials for $UT_n(\mathbb{F})$ when $n\geq 3$ and $p\neq 2$ ; the set of all 
$*$-polynomial identities for $UT_3(\mathbb{F})$ when $\mathbb{F}$ is infinite and $p>2$.
\end{abstract}

\section{Introduction}
Let $\mathbb{F}$ be a field of characteristic $p\neq 2$. In this paper, every algebra is unitary associative over $\mathbb{F}$ and every involution is of the first kind. 

We will talk a little about the involutions of the matrix algebra $M_n(\mathbb{F})$ and its subalgebra 
$UT_n(\mathbb{F})$. There are two important involutions on $M_n(\mathbb{F})$: the transpose and symplectic. When $\mathbb{F}$ is algebraically closed, these are the only involutions  up to isomorphism.
With respect to algebra
$UT_n(\mathbb{F})$, there exist two classes of inequivalent involutions when $n$ is even and a single class otherwise 
(see \cite[Proposition 2.5]{vinkossca}) for all $\mathbb{F}$ (finite or infinite). 

Given two disjoint infinite sets $Y=\{y_1,y_2, \ldots \}$
and $Z=\{z_1,z_2, \ldots \}$, denote by $\mathbb{F} \langle Y \cup Z \rangle$ the free unitary associative algebra, freely
generated  by $Y\cup Z$, with the involution $*$ where 
\[y_i^*=y_i \ \ \mbox{and} \ \ z_i^*=-z_i,\]
for all $i\geq 1$. Given an algebra with involution $(A,\circledast)$, denote by $Id(A,\circledast)$ the set of its $*$-polynomial
identities, that is, the set of all
$f(y_1,\ldots ,y_m,z_1,\ldots ,z_n) \in 
\mathbb{F} \langle Y \cup Z \rangle$ such that 
\[f(a_1,\ldots ,a_m,b_1,\ldots ,b_n)=0 \]
for all $a_1,\ldots ,a_m \in A^+$ and $b_1,\ldots ,b_n \in A^-$. Here, $A^+$   ( $A^-$ ) is the set of all symmetric 
(skew-symmetric) elements  of $A$. 

When we study $Id(M_n(\mathbb{F}), \circledast)$ and $F$ is infinite, it is sufficient to consider the transpose and symplectic involutions  (see \cite[Theorem 3.6.8]{gzbook}). The case $n=2$ was described as follows: 
  Levchenko \cite{levchenkocarzero,levchenkofinito} for $p=0$ or $\mathbb{F}$ finite; 
  Colombo and Koshlukov \cite{colomboplamen} for $\mathbb{F}$ infinite with $p>2$. 
  
  With respect to 
$Id(UT_n(\mathbb{F}), \circledast)$, the case $n=2$ was described as follows: 
  Di Vincenzo, Koshlukov and La Scala \cite{vinkossca} when $\mathbb{F}$ is infinite; Urure
  and Gon\c{c}alves \cite{ronalddimas} when $\mathbb{F}$ is finite. The case $n=3$  also was described in
\cite{vinkossca} when $p=0$. 

The main result of this paper is the description of $Id(UT_3(\mathbb{F}), \circledast)$
for all involutions of the first kind $\circledast$ when $\mathbb{F}$ is infinite and $p>2$  (see Theorem \ref{teoremaprincipal1}). 

Recently, Aljadeff, Giambruno, Karasik (\cite{Aljgiakar}) and Sviridova (\cite{sviridova}) proved that if 
$A$ is an algebra with involution $\circledast$ and $p=0$, then $Id(A, \circledast)$ is finitely generated as a T$(*)$-ideal.
We find a finite generating set 
of $Id(UT_3(\mathbb{F}), \circledast)$ as a T($*$)-ideal when $\mathbb{F}$ is infinite and $p>2$.
It is the same of the case $p=0$ (see Theorem \ref{teoremaprincipal1} and \cite[Theorem 6.6]{vinkossca}). 

Given an algebra with involution $(A,\circledast)$, denote by $C(A,\circledast)$ 
the set of its $*$-central polynomials, that is,
the set of all
$f(y_1,\ldots ,y_m,z_1,\ldots ,z_n) \in 
\mathbb{F} \langle Y \cup Z \rangle$ such that 
\[f(a_1,\ldots ,a_m,b_1,\ldots ,b_n) \in Z(A) \]
for all $a_1,\ldots ,a_m \in A^+$ and $b_1,\ldots ,b_n \in A^-$. Here, $Z(A)$ is the center of $A$. 

If $\mathbb{F}$ is infinite, then Brand\~{a}o and Koshlukov \cite{brandaoplamen} described  $C(M_2(\mathbb{F}), \circledast)$.
For every $\mathbb{F}$ (finite and infinite), Urure
  and Gon\c{c}alves \cite{ronalddimascentral2} described  $C(UT_2(\mathbb{F}), \circledast)$. Differently of central 
  polynomials, there exists non trivial $*$-central 
  polynomial for $UT_2(\mathbb{F})$. But this is not true in general. In this paper, we prove that if $n\geq 3$
  then 
  \[C(UT_n(\mathbb{F}), \circledast)=Id(UT_n(\mathbb{F}), \circledast)+\mathbb{F}\]
  for all $\mathbb{F}$ and $\circledast$.

\section{Involution} \label{secaocomasduasinvolucoes}

We suggest to the reader to see Section 2, page 546 and Section 5 of \cite{vinkossca}. We will use several results 
from there. 

Given $n\geq 1$, let $J \in M_n(\mathbb{F})$ and $D \in M_{2m}(\mathbb{F})$ (if $n=2m$) be the following matrices:
\[J=\left[\begin{array}{cccc}
0 & \dots & 0 & 1 \\
0 & \dots & 1 & 0 \\
\vdots & \udots & \vdots & \vdots \\
1 & \dots & 0 & 0 \\
\end{array}\right] \ \ \mbox{and} \ \ 
D=\left[\begin{array}{cc}
I_m & 0 \\
0 & -I_m \\
\end{array}\right],\]
where $I_m$ is the identity matrix. Define the maps 
$\ast: UT_n(\mathbb{F}) \rightarrow UT_n(\mathbb{F})$ and $s: UT_{n}(\mathbb{F}) \rightarrow UT_{n}(\mathbb{F})$
(if $n$ is even)
by
\[A^{\ast}=JA^tJ \ \ \mbox{and} \ \ A^s=DA^{\ast}D,\]
where $A^t$ is the transpose matrix of $A$. We known that $*$ and $s$ are involutions on  $UT_{n}(\mathbb{F})$. 
Moreover:
\begin{itemize}
\item[a)] The involution $*$ is not equivalent to  $s$.

\item[b)] Every involution on $UT_{n}(\mathbb{F})$ is equivalent  either to $*$ or to  $s$.
\end{itemize}
See \cite[Propositions 2.5 and 2.6]{vinkossca} for details. In particular, we have the following corollary:

\begin{corollary} \label{corolarioinvolucoestresportres}
If $\circledast$ is an involution on $UT_{3}(\mathbb{F})$ then $\circledast$ is equivalent  to $*$,
where
\begin{equation}\label{involucaotresportres}
\left[
\begin{array}{ccc}
a_{11}&a_{12}&a_{13} \\
0&a_{22}&a_{23} \\
0&0&a_{33}
\end{array}
\right]^*
=
\left[
\begin{array}{ccc}
a_{33}&a_{23}&a_{13} \\
0&a_{22}&a_{12} \\
0&0&a_{11}
\end{array}
\right].
\end{equation}
Moreover, $Id(UT_3(\mathbb{F}), \circledast)=Id(UT_3(\mathbb{F}), *)$.
\end{corollary}

\section{$\ast$-Polynomial Identities for $UT_3(\mathbb{F})$}

Let $*$ be the involution on $UT_3(\mathbb{F})$ defined in (\ref{involucaotresportres}). From now on $\mathbb{F}$
is an infinite field of characteristic $p>2$. We denote 
\[UT_3(\mathbb{F})=UT_3 \ \mbox{and} \ Id(UT_3(\mathbb{F}),\ast )=Id.\]
In this section we will describe $Id$.

The vector spaces of symmetric and skew-symmetric elements of $UT_3$ are respectively 
\[UT_3^+=\mbox{span} \ \{e_{11}+e_{33}, \ e_{22}, \ e_{12}+e_{23}, \ e_{13}\} \ \ \mbox{and} \ \ 
UT_3^-=\mbox{span} \ \{e_{11}-e_{33},  \ e_{12}-e_{23}\} .\]
Thus, we have the following lemma.

\begin{lemma}\label{skewsymmetricine13}
If $f(y_1,\ldots,y_n,z_1,\ldots,z_m)\in \mathbb{F}\langle Y\cup Z \rangle $ and 
\[f(a_1,\ldots,a_n,b_1,\ldots,b_m) \in \mbox{span}\{e_{13}\}\]
for all $a_1,\ldots,a_n \in UT_3^+$, \ $b_1,\ldots,b_m \in UT_3^-$, then
\[(f-f^*) \in Id.\]
\end{lemma}

\begin{proof}
Since $f-f^*$ is skew-symmetric we have that $(f-f^*)(a_1,\ldots,a_n,b_1,\ldots,b_m)$ is skew-symmetric.
But $(f-f^*)(a_1,\ldots,a_n,b_1,\ldots,b_m)=\alpha e_{13}$ is symmetric,  where $\alpha \in \mathbb{F}$. Thus $\alpha =0$ and
$(f-f^*) \in Id$.
\end{proof}

If $f\in \mathbb{F}\langle Y\cup Z \rangle^+$ and $g\in \mathbb{F}\langle Y\cup Z \rangle^-$, we denote
$|f|=1$ and $|g|=0$.
Thus, if $h \in \mathbb{F}\langle Y\cup Z \rangle^+ \cup \mathbb{F}\langle Y\cup Z \rangle^-$ then
\begin{equation}\label{hstar}
h^*=-(-1)^{|h|}h.
\end{equation}
From now on, we denote by $x_i$ any element of $ \{ y_i , z_i \} $ and write $|[x_i,x_j]|=|x_ix_j|$. Here,
\[[x_i,x_j]=x_ix_j-x_jx_i \ \ \mbox{and} \ \ [x_1,\ldots , x_n]=[[x_1,\ldots , x_{n-1}],x_n]\]
are the commutators.

\begin{proposition}\label{ident}
The following polynomials belong to $Id$:
\begin{itemize}
\item[(i)] $s_3(z_1,z_2,z_3)=z_1[z_2,z_3]-z_2[z_1,z_3]+z_3[z_1,z_2]$,
\item[(ii)] $(-1)^{|x_1x_2|}[x_1,x_2][x_3,x_4]-(-1)^{|x_3x_4|}[x_3,x_4][x_1,x_2]$,
\item[(iii)] $(-1)^{|x_1x_2|}[x_1,x_2][x_3,x_4]-(-1)^{|x_1x_3|}[x_1,x_3][x_2,x_4] +(-1)^{|x_1x_4|}[x_1,x_4][x_2,x_3]$,
\item[(iv)] $z_1[x_3,x_4]z_2+(-1)^{|x_3x_4|}z_2[x_3,x_4]z_1$,
\item[(v)] $[x_1,x_2]z_5[x_3,x_4]$,
\item[(vi)] $z_1[x_4,x_5]z_2x_3+(-1)^{|x_3|}x_3z_1[x_4,x_5]z_2$.
\end{itemize}
\begin{proof}
Since $s_3$ is the standard polynomial and $\dim UT_3^-=2$ it follows that $s_3 \in Id$.

By (\ref{hstar}), the polynomial (ii) has the form $f-f^*$ where $f=(-1)^{|x_1x_2|}[x_1,x_2][x_3,x_4]$. Thus, by Lemma \ref{skewsymmetricine13}, it is a $*$-identity for $UT_3$. 

Defining $f=z_1[x_3,x_4]z_2$, we can use the same argument as in (ii) to prove that (iv) belongs to $Id$.

Defining $f=z_1[x_4,x_5]z_2x_3$, we can use (\ref{hstar}), Lemma \ref{skewsymmetricine13} and (iv) to prove that 
(vi) belongs to $Id$.

The proof that (iii) and (v) are $*$-identities for $UT_3$ consists of a straightforward 
verification. 
\end{proof}
\end{proposition}

\begin{notation}
From now on, we denote by $I$ the $T(\ast)$-ideal generated by the polynomials of Proposition \ref{ident}. We will deduce some consequences of these identities.
\end{notation}

\begin{lemma} \label{relacoes1}
The following polynomials belong to $I$:
\begin{itemize}
\item[i)] $[x_1,x_2][x_3,x_4][x_5,x_6]$,
\item[ii)] $[x_1,x_2][x_3,x_4]x_5+(-1)^{|x_5|}x_5[x_1,x_2][x_3,x_4]$,
\item[iii)] $[x_1,x_2,x_5][x_3,x_4]-(-1)^{|x_5|}[x_1,x_2][x_3,x_4,x_5]$.
\end{itemize}
\begin{proof}
The proof of this lemma is similar to the proof of  Lemma 5.2, Lemma 5.3 and Lemma 5.5 in \cite{vinkossca}. 
\end{proof}
\end{lemma}

\begin{lemma}\label{relacoes2}
Consider the quotient algebra $\mathbb{F}\langle Y,Z \rangle / I$. If $\sigma \in \mbox{Sym}(n)$ and $\rho \in \mbox{Sym}(m)$ then:
\begin{itemize}
\item[a)] $[z_a,z_b,z_{\sigma(1)},\dots,z_{\sigma(n)}]+I=[z_a,z_b,z_1,\dots,z_n]+I$,
\item[b)] $z_{\rho(1)}\dots z_{\rho(m)}[x_a,x_b,x_{\sigma(1)},\dots ,x_{\sigma(n)}][x_c,x_d]+I= \\ z_1\dots z_m[x_a,x_b,x_1,\dots,x_n][x_c,x_d]+I$,
\item[c)] $[x_a,x_b,x_{\sigma(1)},\dots ,x_{\sigma(n)},y_i]+I=[x_a,x_b,x_1,\dots ,x_n,y_i]+I$,
\item[d)] $x_{\sigma(1)}\dots x_{\sigma(n)}z_j[x_a,x_b]+I=x_1\dots x_nz_j[x_a,x_b]+I$,
\item[e)] $[x_a,x_b]z_jx_{\sigma(1)}\dots x_{\sigma(n)}+I=[x_a,x_b]z_jx_1\dots x_n+I$.
\end{itemize}
\begin{proof}
Here, $\mbox{Sym}(n)$ 
is the symmetric group of $\{1,\ldots , n\}$. The proof of this lemma is similar to the proof of  Remark 5.4 and Remark 5.19 in \cite{vinkossca}.  
\end{proof}
\end{lemma}

The next result is the Lemma 5.6 of \cite{vinkossca}. In the page 554 line 6 there is a little error and we correct it bellow.

\begin{lemma}\label{produtos2comut}
For all $n\geq 0$, the following polynomial belongs to $I$:
\begin{align*}
& (-1)^{|x_4x_3|}[x_4,x_3,x_{i_1},\dots ,x_{i_n}][x_2,x_1]\\ - & (-1)^{|x_4x_2|}[x_4,x_2,x_{i_1},\dots ,x_{i_n}][x_3,x_1] \\
  + & (-1)^{|x_3x_2|}[x_3,x_2,x_{i_1},\dots ,x_{i_n}][x_4,x_1].
\end{align*}

\begin{proof}
The proof is by induction on $n$. Note that it suffice to prove that the following polynomial is in $I$:
\begin{align*}
 g & =(-1)^{|x_4x_3|}[x_4,x_3]x_5[x_2,x_1] \\
& -(-1)^{|x_4x_2|}[x_4,x_2]x_5[x_3,x_1] \\
& +(-1)^{|x_3x_2|}[x_3,x_2]x_5[x_4,x_1].
\end{align*}

If $x_5$ is skew-symmetric then $g \in I$ by Proposition \ref{ident} - (v). 

Suppose that $x_5$ is symmetric and denote $x_5=y_5$. Write 
\begin{align*}
f(x_1,x_2,x_3,x_4) & = (-1)^{|x_4x_3|}[x_4,x_3][x_2,x_1] \\
 & -(-1)^{|x_4x_2|}[x_4,x_2][x_3,x_1] \\
 & +(-1)^{|x_3x_2|}[x_3,x_2][x_4,x_1].
 \end{align*}

By the identities (ii) and (iii) of Proposition \ref{ident}, we have that $f\in I$ and therefore 
$f(y_5x_1,x_2,x_3,x_4) \in I$. Now we use the equality 
\[ [x_i,y_5x_1]=y_5[x_i,x_1]+[x_i,y_5]x_1 \]
to finish the proof.
\end{proof}
\end{lemma}

\begin{lemma}\label{relacoes3}
The following polynomials belong to $I$:
\begin{itemize}
\item[i)] $[z_1,z_2][x_3,x_4]-z_1[x_3,x_4]z_2$, when $|x_3|=|x_4|$.
\item[ii)] $[z_1,z_2][z_3,y_4]+z_1[z_2,y_4]z_3-z_2[z_1,y_4]z_3$.
\end{itemize}
\begin{proof}
The proof of this lemma is similar to the proof of  Lemma 5.10 in \cite{vinkossca}.
\end{proof}
\end{lemma}

\begin{lemma}\label{relacoes4}
For all $n\geq 3$, the following polynomial belongs to $I$:
\begin{equation*}
f_n=z_1[z_3,z_2,x_4,\dots ,x_n]-z_2[z_3,z_1,x_4,\dots ,x_n]+z_3[z_2,z_1,x_4,\dots ,x_n].
\end{equation*}
\begin{proof}
The proof of this lemma is similar to the proof of  Lemma 5.11 in \cite{vinkossca}.
\end{proof}
\end{lemma}

\begin{lemma} \label{relacoes5}
For all $m\geq 2$, the following polynomials are elements of $I$:
\begin{itemize}
\item[a)] $z_1z_2[y_1,y_2,\dots ,y_m]-z_2[y_2,z_1,y_1,y_3,\dots ,y_m]+z_2[y_1,z_1,y_2,\dots ,y_m]$.
\item[b)] $z_1z_2[y_1,z_3,y_2,\dots ,y_m]+z_2[y_1,z_3,z_1,y_2,\dots ,y_m]$.
\item[c)] $z_1z_2[z_3,z_4,y_1,\dots ,y_m]+z_2[z_3,z_4,z_1,y_1,\dots ,y_m]$.
\end{itemize}
\begin{proof} 
The proof of this lemma is similar to the proof of  Lemma 5.13, Lemma 5.14 and  Lemma 5.15 in \cite{vinkossca}.
\end{proof}
\end{lemma}

Let $B$ be the subspace of $\mathbb{F}\langle Y,Z \rangle$ formed by all $Y$-proper polynomials.
Since $\mathbb{F}$ is an infinite field, it is known that  
$Id$ and $I$ are generated, as a T$(*)$-ideals, by its multihomogeneous elements in $B$.  
See \cite[Lemma 2.1]{drenskygiambruno} and \cite[Page 546]{vinkossca} for details. 

If $M=(m_1,\dots, m_k)$ and $N=(n_1,\dots ,n_s)$, denote by $B_{MN}$ 
the following multihomogeneous subspace of $B$:
\[B_{MN}=\{f(y_1,\ldots ,y_k,z_1,\ldots ,z_s) \in B : \ \deg_{y_i} f=m_i, \ \deg_{z_j}f=n_j, \ 1\leq i \leq k , \  1\leq j \leq s\}.\]
We shall prove that $I=Id$, that is, $I\cap B_{MN}= Id\cap B_{MN}$ for all $M,N$.

\begin{observation} \label{observacaoordemdasvariaveis}
Let $\sigma \in \mbox{Sym}(k)$ and $\rho \in \mbox{Sym}(s)$. Since
the T$(*)$-ideals generated by 
\[f(y_1,\ldots ,y_k,z_1,\ldots ,z_s) \ \ \mbox{and} \ \ f(y_{\sigma (1)},\ldots ,y_{\sigma (k)},z_{\rho (1)},\ldots ,z_{\rho (s)})\]
are equal it is sufficient to prove $I\cap B_{MN}= Id\cap B_{MN}$ for 
\begin{equation} \label{aaa}
1\leq m_1 \leq \ldots \leq m_k \ \ \mbox{and} \ \ 1\leq n_1 \leq \ldots \leq n_s. 
\end{equation}
From now on we assume (\ref{aaa}).
\end{observation} 
 
Denote 
\[B_{MN}(I)=B_{MN}/I \cap B_{MN}  \ \mbox{and} \ B_{MN}(Id)=B_{MN}/ Id \cap B_{MN} .\]
When  $m_1= \dots = m_k=n_1 = \dots = n_s=1$, we write $B_{MN}=\Gamma_{ks}$.

Suppose $m_1, \dots , m_k \geq 1$ and $n_1 , \dots , n_s \geq 1$. Write $m_1+ \dots + m_k=m$ and 
$n_1 + \dots + n_s =n$. Let $\varphi_{MN} : \Gamma_{mn}(I) \rightarrow B_{MN}(I)$ be the linear map defined by
\[\varphi_{MN}(f(y_1,\ldots ,y_m,z_1, \ldots , z_n)+I\cap \Gamma_{mn})=\] 
\[f(\underbrace{y_1,\ldots ,y_1}_{m_1},\ldots ,\underbrace{y_k,\ldots , y_k}_{m_k},\underbrace{ z_1, \ldots , z_1}_{n_1}
, \ldots , \underbrace{z_s, \ldots , z_s}_{n_s})+I\cap B_{MN}.\]
Since $\varphi_{MN}$ is onto, we have the following proposition:
\begin{proposition}\label{relacaodegammaeb}
Consider the above notations. If the vector space  $\Gamma_{mn}(I)$ is spanned by a subset $S$, then $B_{MN}(I)$ is spanned by $\varphi_{MN}(S)$. 
\end{proposition}

Fix the following order on $Y\cup Z$ :
\[z_1< z_2 < \ldots < y_1< y_2 < \ldots  \]

\begin{definition}
Let $S_1$ be the set of all polynomials
\[f=z_{i_1}\dots z_{i_t}[x_{j_1},\dots ,x_{j_l}][x_{k_1},\dots ,x_{k_q}]\]
where $t, \ l, \ q \geq 0$,    $l \neq 1$, $ q \neq 1$,     $z_{i_1} \leq \ldots \leq z_{i_t}$,   $x_{j_1} > x_{j_2} \leq \ldots \leq x_{j_l}$ and  
$x_{k_1} > x_{k_2} \leq \ldots \leq x_{k_q}$.   We say that $f$ is an $S_1$-standard polynomial. 
\end{definition}

\begin{definition}
Let $S_2 \subset S_1$ be the set of all polynomials
\[f=z_{i_1}\dots z_{i_t}[x_{j_1},\dots ,x_{j_l}][x_{k_1},\dots ,x_{k_q}] \in S_1\]
such that: if $l\geq 2$ then $q=0$ or $q=2$, and when $q=2$ we have
that $x_{j_1}\geq x_{k_1}$ and $x_{j_2}\geq x_{k_2}$. If $f\in S_2$ we say that $f$ is an $S_2$-standard polynomial.
\end{definition} 

\begin{proposition}\label{geradoresdebmni}
The vector space $B_{MN}(I)$ is spanned by the set of all elements $f+I \cap B_{MN}$ where $f \in B_{MN}$ is $S_2$-standard. 
\begin{proof}
This proposition is true for $\Gamma_{mn}(I)$. 
In fact, we can use the same proof as in \cite[Proposition 5.8]{vinkossca}. Now, 
if $x_i < x_j$ then $\varphi_{MN}(x_i) \leq  \varphi_{MN}(x_j)$.
Thus, by Proposition \ref{relacaodegammaeb}, the general case is proved. 
\end{proof}
\end{proposition}

\begin{observation}\label{ordemnasvariaveis}
Let $F[y_{ij}^k , z_{ij}^k]=F[y_{ij}^k , z_{ij}^k \ : \ i,j,k \geq 1]$ 
be the free commutative algebra freely generated by the set of variables $L=\{y_{ij}^k , z_{ij}^k\ : \ i,j,k \geq 1\}$. Given an order $>$ on $L$, consider the order on the monomials of $F[y_{ij}^k,\ z_{ij}^k]$ 
induced by $>$ as follows: if $w_1 \geq w_2 \geq \ldots \geq  w_n , \  w_1' \geq w_2' \geq \ldots \geq w_m'$  are in $L$ then
\[w_1w_2 \ldots w_n > w_1'w_2' \ldots w_m'\]
if and only if 
\begin{itemize}
\item either $w_1=w_1', \ \ldots , w_l=w_l', \ w_{l+1}>w_{l+1}'$ for some $l$,
\item or $w_1=w_1', \  \ldots , w_m=w_m'$ and $n>m$.
\end{itemize}
 Given
$f\in F[y_{ij}^k,  z_{ij}^k]$, we denote by $m(f)$ its leading monomial. 

In $UT_3(F[y_{ij}^k,  z_{ij}^k])$ consider the \textit{qgeneric matrices} 
\[
Z_k=\left[
\begin{array}{ccc}
z_{11}^k & z_{12}^k & 0\\
0 & 0 & -z_{12}^k\\
0 & 0 & -z_{11}^k
\end{array}
\right] \ \ \mbox{and} \ \ 
Y_k=\left[
\begin{array}{ccc}
y_{11}^k & y_{12}^k & y_{13}^k\\
0 & 0 & y_{12}^k\\
0 & 0 & y_{11}^k
\end{array}
\right].
\]
Note, the $(2,2)$-entry of $Y_k$ is $0$.
\end{observation}

By using an analogous argument to the  \cite[Lemma 6.1]{vinkossca} we obtain the next lemma.

\begin{lemma}\label{lemamatrizesgenericas}
 If $f(y_1,\ldots , y_k,z_1, \ldots , z_s) \in Id$, then $f(Y_1,\ldots , Y_k,Z_1, \ldots , Z_s) =0$.
\end{lemma}

\subsection{Subspaces $B_{MN}$ where $N=(0)$}

If $M=(m_1,\dots, m_k)$ and 
$N=(0)$, then denote $B_{MN}=B_{M0}$, that is
\[B_{M0}=\{f(y_1,\ldots,y_k) \in B \ : \ \deg_{y_i} f=m_i, \ 1\leq i \leq k \}.\]

A polynomial in $S_2\cap B_{M0}$ has the form
\begin{equation}\label{aa11}
f^{(i_1)}=[y_{i_1},y_{i_2},\ldots,y_{i_m}] \ \mbox{or} \ f^{(i_1,j_1)}=[y_{i_1},y_{i_2},\ldots,y_{i_{m-2}}][[y_{j_1},y_{j_2}]
\end{equation}
where $i_1>i_2\leq i_3 \leq \ldots \leq i_m$ for $f^{(i_1)}$;  $i_1>i_2\leq i_3 \leq \ldots \leq i_{m-2},$ \ $i_1\geq j_1 >j_2$ and $i_2\geq j_2$ for $ f^{(i_1,j_1)}$. 

\begin{proposition}\label{propositionmain1}
The set $\{f+Id\cap B_{M0} \ : \ f\in S_2\cap B_{M0}\}$ is a basis for the vector space $B_{M0}(Id)$.
In particular, $Id\cap B_{M0}=I \cap B_{M0}$.
\begin{proof}
Since $I\subseteq Id$ we have, by Proposition \ref{geradoresdebmni}, 
\[B_{M0}(Id)=\mbox{span} \{f+Id\cap B_{M0} \ : \ f\in S_2\cap B_{M0}\}.\]

We will use the notations of Observation \ref{ordemnasvariaveis}. 
Consider some order $>$ on $L$ such that 
\[y_{12}^{i+1} > y_{12}^{i} >y_{11}^{i+1} > y_{11}^{i} \]
for all $i\geq 1$. By using the qgeneric matrices we have the following equalities:
\begin{itemize}
\item[(a)]  $ \displaystyle [Y_{1},Y_{2},\ldots , Y_{2l}]=(y_{11}^1y_{12}^2-y_{11}^2y_{12}^1)\left(\prod_{s=3}^{2l} y_{11}^{s}\right)(e_{12}-e_{23})$,

\item[(b)] $\displaystyle [Y_{1},Y_{2},\ldots , Y_{2l+1}]=-(y_{11}^1y_{12}^2-y_{11}^2y_{12}^1)\left(\prod_{s=3}^{2l} y_{11}^{s}\right)(
y_{11}^{2l+1}(e_{12}+e_{23})-2y_{12}^{2l+1}e_{13})$, 

\item[(c)] $\displaystyle [Y_{1},Y_{2},\ldots , Y_{2l-2}][Y_{2l-1}, Y_{2l}]=(y_{11}^1y_{12}^2-y_{11}^2y_{12}^1)\left(\prod_{s=3}^{2l-2} y_{11}^{s}\right)(y_{11}^{2l-1}y_{12}^{2l}-y_{11}^{2l}y_{12}^{2l-1})e_{13}$, 

\item[(d)]  $\displaystyle [Y_{1},Y_{2},\ldots , Y_{2l-1}][Y_{2l}, Y_{2l+1}]=-(y_{11}^1y_{12}^2-y_{11}^2y_{12}^1)\left(\prod_{s=3}^{2l-1}
 y_{11}^{s}\right)
(y_{11}^{2l}y_{12}^{2l+1}-y_{11}^{2l+1}y_{12}^{2l})e_{13}.$
\end{itemize}

Let $f^{(i_1)}$, $f^{(i_1,j_1)}$ as in (\ref{aa11}). 
 Write
\[f^{(i_1)}(Y_1,\ldots , Y_k)=\sum f^{(i_1)}_{ij}e_{ij} \ \mbox{and} \ 
f^{(i_1,j_1)}(Y_1,\ldots , Y_k)=\sum f^{(i_1,j_1)}_{ij}e_{ij}.
\]

We shall prove that $\{f+Id\cap B_{M0} \ : \ f\in S_2\cap B_{M0}\}$ is a linearly independent set of $B_{M0}(Id)$.
Suppose 
\[ \sum \alpha_{i_1}f^{(i_1)}+\sum \alpha_{i_1,j_1}f^{(i_1,j_1)} \in Id,
\]
where $\alpha_{i_1} , \alpha_{i_1,j_1} \in \mathbb{F}$.
By Lemma \ref{lemamatrizesgenericas} we have 
\[ \sum \alpha_{i_1}f^{(i_1)}(Y_1,\ldots , Y_k)+\sum \alpha_{i_1,j_1}f^{(i_1,j_1)}(Y_1,\ldots , Y_k)=0
.\]
The leading monomials of $f^{(i_1)}_{12}$ and $f^{(i_1,j_1)}_{13}$ are
\[m(f^{(i_1)}_{12})=y_{12}^{i_1}\left(\prod_{s=2}^{m} y_{11}^{i_s}\right) \ \mbox{and} \ 
m(f^{(i_1,j_1)}_{13})=y_{12}^{i_1}y_{12}^{j_1}\left(\prod_{s=2}^{m-2} y_{11}^{i_s}\right)y_{11}^{j_2}.
\]
Moreover, the coefficients of $m(f^{(i_1)}_{12})$ and $m(f^{(i_1,j_1)}_{13})$ 
in $f^{(i_1)}_{12}$ and $f^{(i_1,j_1)}_{13}$
are $\alpha_{i_1}$ and $\alpha_{i_1,j_1}$ respectively.

Since $f^{(i_1,j_1)}_{12}=0$, we have $ \sum \alpha_{i_1}f^{(i_1)}_{12}=0$. Thus, the coefficient of the maximal monomial 
of the set $\{m(f^{(i_1)}_{12}) \ : \ i_1 \geq 1\}$ is $0$. By induction, every coefficient $\alpha_{i_1}$ is $0$. This implies 
$ \sum \alpha_{i_1,j_1}f^{(i_1j_1)}_{13}=0$ and we can use similar argument to prove that every $\alpha_{i_1,j_1}$ is $0$. 
\end{proof}
\end{proposition}

\subsection{Subspaces $B_{MN}$ where $M=(0)$}

If 
$M=(0)$ and $N=(n_1,\dots ,n_s)$, then denote $B_{MN}=B_{0N}$, that is
\[B_{0N}=\{f(z_1,\ldots,z_s) \in B \ : \ \deg_{z_i} f=n_i, \ 1\leq i \leq s \}.\]

\begin{definition}
Let $S_3 \subset S_2$ be the set of all polynomials $f,f^{(j_1)},f^{(i_1,j_1)} \in B_{0N}$ such that:
\begin{align}
& \bullet f=z_1^{n_1}\dots z_s^{n_s} \in S_2,  \nonumber \\ 
& \bullet f^{(j_1)}=[z_{j_1},z_{j_2},\dots ,z_{j_n}] \in S_2, \label{caradospolf}\\
& \bullet f^{(i_1,j_1)}=z_{i_1}[z_{j_1},\dots ,z_{j_{n-1}}] \in S_2 \ \mbox{ where } i_1 \leq j_1. \nonumber
\end{align} 
If $f\in S_3$ we say that $f$ is an $S_3$-standard polynomial.
\end{definition} 

\begin{proposition}\label{geradoresdebzeroni}
The vector space $B_{0N}(I)$ is spanned by the set of all elements $f+I \cap B_{0N}$ where $f \in B_{0N}$ is $S_3$-standard. 
\begin{proof}
This proposition is true for $\Gamma_{0n}(I)$. 
In fact, we can use the same proof as in \cite[Proposition 5.12]{vinkossca}. Now, 
if $z_i < z_j$ then $\varphi_{0N}(z_i) \leq  \varphi_{0N}(z_j)$.
Thus, by Proposition \ref{relacaodegammaeb}, the general case is proved. 
\end{proof}
\end{proposition}

\begin{proposition} \label{propositionmain2}
The set $\{f+Id\cap B_{0N} \ : \ f\in S_3\}$ is a basis for the vector space $B_{0N}(Id)$.
In particular, $Id\cap B_{0N}=I \cap B_{0N}$.
\begin{proof}
Since $I\subseteq Id$ we have, by Proposition \ref{geradoresdebzeroni}, 
\[B_{0N}(Id)=\mbox{span} \{f+Id\cap B_{0N} \ : \ f\in S_3\}.\]

We will use the notations of Observation \ref{ordemnasvariaveis}. 
Consider some order $>$ on $L$ such that
\[z_{12}^{i+1} > z_{12}^{i} >z_{11}^{i+1} > z_{11}^{i} \]
for all $i\geq 1$.  
By using the qgeneric matrices we have the following equalities:
\begin{itemize}
\item[(a)] $ \displaystyle [Z_1,Z_2,\dots ,Z_n]= (-1)^n(z_{11}^1z_{12}^2-z_{12}^1z_{11}^2)\left(\prod_{s=3}^{n}
 z_{11}^s\right)(e_{12}-e_{23})$,

\item[(b)] $ \displaystyle
Z_1[Z_2,Z_3,\dots ,Z_n]= (-1)^{n-1}(z_{11}^2z_{12}^3-z_{12}^2z_{11}^3)\left(\prod_{s=4}^n
 z_{11}^s\right)(z_{11}^1e_{12}-z_{12}^1e_{13})$.
\end{itemize}

Let $f, \ f^{(j_1)}$ and $f^{(i_1,j_1)}$ as in $(\ref{caradospolf})$. 
Write $f(Z_1,\ldots , Z_s)=\sum f_{ij}e_{ij}$,
\[f^{(j_1)}(Z_1,\ldots , Z_s)=\sum f^{(j_1)}_{ij}e_{ij} \ \mbox{and} \ 
f^{(i_1,j_1)}(Z_1,\ldots , Z_s)=\sum f^{(i_1,j_1)}_{ij}e_{ij}.
\]

Suppose
\[ \alpha f+\sum \alpha_{j_1}f^{(j_1)}+\sum \alpha_{i_1,j_1}f^{(i_1,j_1)} \in Id,
\]
where $\alpha, \alpha_{j_1}, \alpha_{i_1,j_1} \in \mathbb{F}$. By Lemma \ref{lemamatrizesgenericas} we have 
\[ \alpha f(Z_1,\ldots , Z_s)+\sum \alpha_{j_1}f^{(j_1)}(Z_1,\ldots , Z_s)+\sum \alpha_{i_1,j_1}f^{(i_1,j_1)}(Z_1,\ldots , Z_s)=0.\]

Since $f^{(j_1)}_{11}=f^{(i_1,j_1)}_{11}=0$, we obtain $\alpha f_{11}=0$ and so $\alpha =0$. 
Note that 
\[m(f^{(j_1)}_{23})=z_{12}^{j_1}\left(\prod_{l=2}^{n} z_{11}^{j_l}\right) \ \mbox{and} \ 
m(f^{(i_1,j_1)}_{13})=z_{12}^{j_1}z_{12}^{i_1} \left(\prod_{l=2}^{n-1} z_{11}^{j_l}\right).
\]
Moreover, the coefficients of $m(f^{(j_1)}_{23})$ and $m(f^{(i_1,j_1)}_{13})$ 
in $f^{(j_1)}_{23}$ and $f^{(i_1,j_1)}_{13}$
are $\pm \alpha_{j_1}$ and $\pm \alpha_{i_1,j_1}$ respectively.

Since $f^{(i_1,j_1)}_{23}=0$, we have $ \sum \alpha_{j_1}f^{(j_1)}_{23}=0$. Thus, the coefficient of the maximal monomial 
of the set $\{m(f^{(j_1)}_{23}) \ : \ j_1 \geq 1\}$ is $0$. By induction, every coefficient $\alpha_{j_1}$ is $0$. We can use  similar argument to prove that every $\alpha_{i_1,j_1}$ is $0$. 
\end{proof}
\end{proposition}

\subsection{Subspaces $B_{MN}$ where $M\neq (0), (1)$ and $N=(1)$}

If 
$M=(m_1,\dots ,m_k) \neq (0), (1)$ and $N=(1)$, then denote $B_{MN}=B_{M1}$, that is,
\[ B_{M1}=\{f(y_1,\dots ,y_k,z_1)\in B:\mbox{deg}_{y_i}f=m_i \  \mbox{ and }\mbox{deg}_{z_1}f=1 , \ 1 \leq i \leq k\}\]
and $m=m_1+\dots +m_k>1$. 

\begin{definition}
Let $S_3\subset S_2$ be the set of all polynomials $f^{(i_1)},g^{(i_1)},f^{(i_1,j_1)}\in B_{M1}$ such that:
\begin{align}
& \bullet f^{(i_1)}=[y_{i_1},z_1,y_{i_2},\dots ,y_{i_m}] \in S_2,  \nonumber \\ 
& \bullet g^{(i_1)}=z_1[y_{i_1},y_{i_2},\dots ,y_{i_{m}}] \in S_2, \label{caradospolf2}\\
& \bullet f^{(i_1,j_1)}=[y_{i_1},y_{i_2},\dots ,y_{i_{m-1}}][y_{j_1},z_1] \in S_2. \nonumber
\end{align} 
If $f\in S_3$, we say that $f$ is an $S_3$-standard polynomial.
\end{definition}

\begin{proposition} \label{geradoresdebmumi}
The vector space $B_{M1}(I)$ is spanned by the set of all elements $f+I\cap B_{M1}$ where $f\in B_{M1}$ is $S_3$-standard.
\begin{proof}
This proposition is true for $\Gamma_{m1}(I)$. In fact, we can use the same proof as in \cite[Proposition 5.17]{vinkossca}. Now, if $y_i<y_j$ then $\varphi_{M1}(y_i)\leq \varphi_{M1}(y_j)$. Thus, by Proposition \ref{relacaodegammaeb}, the general case is proved.
\end{proof}
\end{proposition}

\begin{proposition} \label{propositionmain3}
The set $\{f+Id\cap B_{M1}:f\in S_3\}$ is a basis for the vector space $B_{M1}(Id)$. In particular, $Id\cap B_{M1}=I\cap B_{M1}$.
\begin{proof}
Since $I\subset Id$, we have, by Proposition \ref{geradoresdebmumi}, 
\[ B_{M1}(Id)=\mbox{span}\{f+Id\cap B_{M1}:f\in S_3\}.\]

We will use the notations of Observation \ref{ordemnasvariaveis}. 
Consider some order $>$ on $L$  such that 
\[ y_{12}^{i+1} > y_{12}^{i} >y_{11}^{i+1} > y_{11}^{i}  >      z_{12}^{i+1} > z_{12}^{i} >z_{11}^{i+1} > z_{11}^{i} \]
for all $i\geq 1$. By using the qgeneric matrices we have the following equalities:
\begin{align*}
& [Y_1,Z_1,Y_2,\dots ,Y_{2l}]= -(y_{11}^1z_{12}^1-y_{12}^1z_{11}^1)\left(\prod_{s=2}^{2l}y_{11}^s\right)\left(e_{12}-e_{23} \right) ,\\
& [Y_1,Z_1,Y_2,\dots ,Y_{2l+1}]=  (y_{11}^1z_{12}^1-y_{12}^1z_{11}^1)\left(\prod_{s=2}^{2l}y_{11}^s \right) \left(y_{11}^{2l+1}(e_{12}+e_{23})-2y_{12}^{2l+1}e_{13}\right) ,
\end{align*} \vspace{-.5cm}
\begin{align*}
& Z_1[Y_1,Y_2,\dots ,Y_{2l}]= (y_{11}^1y_{12}^2-y_{12}^1y_{11}^2)\left(\prod_{s=3}^{2l}y_{11}^s\right)(z_{11}^1 e_{12}
 -z_{12}^1e_{13}), \\
& Z_1[Y_1,Y_2,\dots ,Y_{2l+1}]=
-(y_{11}^1y_{12}^2-y_{12}^1y_{11}^2)\left(\prod_{s=3}^{2l}y_{11}^s\right)(z_{11}^1y_{11}^{2l+1}e_{12}+
(-2z_{11}^1y_{12}^{2l+1}+z_{12}^1y_{11}^{2l+1})e_{13}), \
\end{align*} \vspace{-.5cm}
\begin{align*}
& [Y_1,\dots ,Y_{2l-1}][Y_{2l},Z_1]=-(y_{11}^1y_{12}^2-y_{12}^1y_{11}^2)\left(\prod_{s=3}^{2l-1}y_{11}^s\right)\left(y_{11}^{2l}z_{12}^1-y_{12}^{2l}z_{11}^1\right)e_{13}, \\
& [Y_1,\dots ,Y_{2l}][Y_{2l+1},Z_1]=(y_{11}^1y_{12}^2-y_{12}^1y_{11}^2)\left(\prod_{s=3}^{2l}y_{11}^s\right)\left(y_{11}^{2l+1}z_{12}^1-y_{12}^{2l+1}z_{11}^1\right)e_{13}.\
\end{align*}

Let $f^{(i_1)}$, $g^{(i_1)}$, $f^{(i_1,j_1)}$ as in (\ref{caradospolf2}). Write $f^{(i_1)}(Y_1,\dots ,Y_k,Z_1)=\sum f_{ij}^{(i_1)}e_{ij}$, $g^{(i_1)}(Y_1,\dots ,Y_k,Z_1)=\sum g_{ij}^{(i_1)}e_{ij}$ and $f^{(i_1,j_1)}(Y_1,\dots ,Y_k,Z_1)=\sum f_{ij}^{(i_1,j_1)}e_{ij}$.

Suppose
\[\sum \alpha_{i_1}f^{(i_1)}+\sum \beta_{i_1}g^{(i_1)}+\sum \alpha_{i_1,j_1}f^{(i_1,j_1)}\in Id,\]
where $\alpha_{i_1},\beta_{i_1}, \alpha_{i_1,j_1} \in \mathbb{F}$.

Note that
\[ m(f_{23}^{(i_1)})=y_{12}^{i_1}\left(\prod_{s=2}^my_{11}^{i_s}\right)z_{11}^1,\]
and its coefficient in $f_{23}^{(i_1)}$ is $- \alpha_{i_1}$.
Since $g_{23}^{(i_1)}=f_{23}^{(i_1,j_1)}=0$, we have $\sum \alpha_{i_1}f_{23}^{(i_1)}=0$. Thus, the coefficient of the maximal monomial of the set $\{m(f_{23}^{(i_1)}):i_1\geq 1\}$ is 0. By induction, every coefficient $\alpha_{i_1}$ is 0.
Now, 
\[ m(g_{12}^{(i_1)})=y_{12}^{i_1}\left(\prod_{s=2}^my_{11}^{i_s}\right)z_{11}^1, \]
and its coefficient in $g_{12}^{(i_1)}$ is $\pm \beta_{i_1}$.
Since $f_{12}^{(i_1,j_1)}=0$, we have $\sum \beta_{i_1}g_{12}^{(i_1)}=0$. Thus, the coefficient of the maximal monomial of the set $\{m(g_{12}^{(i_1)}):i_1>1\}$ is 0. By induction, every coefficient $\beta_{i_1}$ is 0. 
This implies $\sum \alpha_{i_1,j_1}f_{13}^{(i_1,j_1)}=0$. Since 
\[ m(f_{13}^{(i_1,j_1)})=y_{12}^{i_1}y_{12}^{j_1}\left(\prod_{s=2}^{m-1}y_{11}^{i_s}\right)z_{11}^1, \]
and its coefficient in $f_{13}^{(i_1,j_1)}$ is $\pm \alpha_{i_1,j_1}$,
we can use similar argument to prove that every $\alpha_{i_1,j_1}$ is $0$.
\end{proof}
\end{proposition}

\subsection{Subspaces $B_{MN}$ where $M=(1)$ and $N\neq (0)$}

If $M=(1)$ and $N=(n_1,\ldots , n_s) \neq (0)$, then denote $B_{MN}=B_{1N}$, that is
\[ B_{1N}=\{f(y_1,z_1,\dots ,z_s)\in B: \mbox{deg}_{y_1}f=1 \mbox{ and } \mbox{deg}_{z_i}f=n_i, \ 1\leq i \leq s\}. \]

\begin{definition} \label{caradospolf3}
Let $S_3$ be the set of all polynomials $f^{(j)}$, $g^{(j)}$, $h^{(j)}$, $f^{(i,j)} \in B_{1N}$ such that
\begin{itemize}
\item[(a)] $f^{(j)}=z_1^{n_1}\dots z_j^{n_j-1}\dots z_s^{n_s}[y_1,z_j]$, where $1\leq j \leq s$.  
\item[(b)] $g^{(j)}=[y_1,z_j]z_1^{n_1}\dots z_j^{n_j-1}\dots z_s^{n_s}$, where $1\leq j \leq s$. 
\item[(c)] $ h^{(j)}=z_1^{n_1}\dots z_j^{n_j-1}\dots z_s^{n_s-1}[y_1,z_j]z_s$, where $1\leq j \leq s$.
\item[(d)] $f^{(i,j)}=z_1^{n_1}\dots z_i^{n_i-1}\dots z_j^{n_j-1}\dots z_s^{n_s}z_i[y_1,z_j]$, where:
\begin{itemize}
\item[(d1)]  $1\leq i \leq j \leq s$  and  $i<s$ if $n_s > 1$,
\item[(d2)]  $1\leq i \leq j \leq s-1$ if $n_s = 1$.
\end{itemize}
\end{itemize} 
If $f\in S_3$, we say that $f$ is an $S_3$-standard polynomial.
\end{definition}

\begin{proposition} \label{geradB1N}
The vector space $B_{1N}(I)$ is spanned by the set of all elements $f+I\cap B_{1N}$ where $f$ is $S_3$-standard.
\begin{proof}
This proposition is true for $\Gamma_{1n}(I)$. In fact, we can use the same proof as in \cite[Proposition 5.20]{vinkossca}. Now, if $z_i<z_j$ then $\varphi_{1N}(z_i)\leq \varphi_{1N}(z_j)$. Thus, by Proposition \ref{relacaodegammaeb}, the general case is proved.
\end{proof}
\end{proposition}

\subsubsection{Subspace $B_{1N}$ where $n_s>1$}


We start this subsection with the next proposition. 
By using similar arguments as the ones used in  \cite[Theorem 6 in Chapter 4]{bahturinbook} we obtain:

\begin{proposition}\label{proposicaocorpoinfinitopotenciasdep}
 Let $\mathbb{F}$ be an infinite field of $char(\mathbb{F})=p >2$. If $H$ is a $T(*)$-ideal then $H$ is generated, as a $T(*)$-ideal,
 by its multihomogeneous elements $f(y_1,\ldots,y_k,z_1,\ldots,z_s)\in H$ with multidegree 
 $(p^{a_1},\ldots,p^{a_k},p^{b_1},\ldots,p^{b_s})$ where $a_1,\ldots,a_k,b_1,\ldots,b_s \geq 0$.
\end{proposition}

We want to show that $I=Id$ by proving $Id\cap B_{MN}= I\cap B_{MN}$. By the last proposition, 
in this subsection is sufficient to consider the case $1 < n_s=p^{b_s} $. Since $char(\mathbb{F})=p \geq 3$ we have 
$n_s \geq 3$. Thus, from now on, we assume $n_s \geq 3$ in $B_{1N}$.

\begin{definition}
Let $S_4$ be the set of all polynomials $f^{(j)}$, $g^{(j)}$, $h^{(s)}$, $p^{(i,j)} \in B_{1N}$ such that:
\begin{align}
& \bullet f^{(j)},g^{(j)},h^{(s)}\in S_3 \mbox{ as in Definition }\ref{caradospolf3}, \label{caradospolf4}\\ 
& \bullet p^{(i,j)}=z_1^{n_1}\dots z_i^{n_i-1}\dots z_j^{n_j-1}\dots z_s^{n_s-1}[z_s,z_i][y_1,z_j], \ 1\leq i\leq j \leq s \mbox{ and } i<s  \nonumber .
\end{align} 
If $f\in S_4$, we say that $f$ is an $S_4$-standard polynomial.
\end{definition}

\begin{proposition}\label{geradB1N(I)}
If $n_s \geq 3$, then the vector space $B_{1N}(I)$ is spanned by the set of all elements $f+I\cap B_{1N}$ where $f$ is $S_4$-standard.
\begin{proof}
Let $\Lambda=\mbox{span}\{f+I\cap B_{1N}: \ f\in S_4\}$.
By Proposition \ref{geradB1N}, it is enough to prove that
\begin{align*}
& h^{(j)}+I\cap B_{1N}\in \Lambda, \  j<s; \\
& f^{(i,j)}+I\cap B_{1N}\in \Lambda, \ 1\leq i \leq j \leq s \mbox{ and } i<s. 
\end{align*}

By Lemma \ref{relacoes2}-d) it follows that $f^{(i,j)}+I=p^{(i,j)}+f^{(j)}+I$. Thus $f^{(i,j)}+I\cap B_{1N}\in \Lambda$.

Write $h^{(j)}=wz_sz_s[y_1,z_j]z_s$ where $w=z_1^{n_1}\dots z_j^{n_j-1}\dots z_s^{n_s-3}$. By Lemma 
\ref{relacoes3}- ii), Lemma \ref{relacoes2}-d) and Lemma \ref{relacoes1}-ii)
we obtain:
\begin{align*}
h^{(j)}+I & = wz_sz_j[y_1,z_s]z_s+wz_s[z_j,z_s][y_1,z_s]+I\\
                 & =wz_jz_s[y_1,z_s]z_s+w[z_s,z_j][y_1,z_s]z_s-wz_s[z_s,z_j][y_1,z_s]+I \\
                 & =h^{(s)}-wz_s[z_s,z_j][y_1,z_s]-wz_s[z_s,z_j][y_1,z_s]+I \\
                 & =h^{(s)}-2p^{(j,s)}+I.
\end{align*}
Therefore, $h^{(j)}+I\cap B_{1N}\in \Lambda $.
\end{proof}
\end{proposition}

\begin{proposition} \label{baseB1N(Id)}
If $n_s \geq 3$, then 
$\{f+Id\cap B_{1N}: \ f\in S_4\}$ is a basis for the vector space $B_{1N}(Id)$. In particular, $Id\cap B_{1N}=I\cap B_{1N}$.
\begin{proof}
Since $I\subset Id$, we have, by Proposition \ref{geradB1N(I)}
\[B_{1N}(Id)=\mbox{span}\{f+Id\cap B_{1N}:f\in S_4\}. \]

We will use the notations of Observation \ref{ordemnasvariaveis}. 
Consider some order $>$ on $L$  such that 
\[ z_{12}^{l+1} > z_{12}^{l} > z_{12}^{s} >z_{11}^{i+1} > z_{11}^{i}  >      y_{12}^{i+1} > y_{12}^{i} >y_{11}^{i+1} > y_{11}^{i} \]
for all $1 \leq l \leq s-2$ and  $i\geq 1$ . 
 By using the qgeneric matrices we have the following equalities:

\begin{align*}
\bullet & Z_1\dots \hat{Z_j}\dots Z_m[Y_1,Z_j]=\left[\prod_{l=1, \ (l\neq j)}^m z_{11}^l\right](y_{11}^1z_{12}^j-y_{12}^1z_{11}^j)e_{12}+ue_{13}, \\
\bullet & [Y_1,Z_j]Z_1\dots \hat{Z_j}\dots Z_m=(-1)^{m-1}(y_{11}^1z_{12}^j-y_{12}^1z_{11}^j)\left[ \prod_{l=1, \ (l\neq j)}^m z_{11}^l\right]e_{23}+ve_{13}, \\
\bullet & Z_1\dots Z_{s-1}Z_s[Y_1,Z_s]Z_s=-\left[ \prod_{l=1}^s z_{11}^l\right](y_{11}^1z_{12}^s-y_{12}^1z_{11}^s)z_{12}^s e_{13}\\
& +\left(2\left[\prod_{l=1}^s z_{11}^l\right](y_{12}^1z_{12}^s+y_{13}^1z_{11}^s)z_{11}^s
-\left[\prod_{l=1}^{s-1} z_{11}^l\right]z_{12}^s(y_{11}^1z_{12}^s-y_{12}^1z_{11}^s)z_{11}^s\right)e_{13}, \\
\bullet & Z_1\dots \hat{Z_i}\dots \hat{Z_j}\dots Z_s[Z_s,Z_i][Y_1,Z_j]= \\
& \left[\prod_{l=1, \ (l\neq i,j)}^s z_{11}^l\right](z_{11}^sz_{12}^i-z_{12}^sz_{11}^i)(y_{11}^1z_{12}^j-y_{12}^1z_{11}^j)e_{13}.
\end{align*} 
for some polynomials $u,v \in \mathbb{F}[L]$.

Let $f^{(j)}$, $g^{(j)}$, $h^{(s)}$, $p^{(i,j)}$ as in (\ref{caradospolf4}). Write
\begin{align*}
f^{(j)}(Y_1,Z_1,\dots ,Z_s)=\sum f_{ab}^{(j)}e_{ab} \ , \ 
 \ \ &g^{(j)}(Y_1,Z_1,\dots ,Z_s)=\sum g_{ab}^{(j)}e_{ab} \ , \\
h^{(s)}(Y_1,Z_1,\dots ,Z_s)=\sum h_{ab}^{(s)}e_{ab} \ , \  \ \ 
&p^{(i,j)}(Y_1,Z_1,\dots ,Z_s)=\sum p_{ab}^{(i,j)}e_{ab} \ ,
\end{align*} 
and suppose
\[ \sum \alpha_jf^{(j)} + \sum \beta_jg^{(j)}+\gamma h^{(s)}+\sum \beta_{i,j}p^{(i,j)} \in Id, \]
where $\alpha_j, \ \beta_j, \ \gamma , \beta_{i,j} \in \mathbb{F}$. 
Now we use the same arguments as  Propositions \ref{propositionmain1},  \ref{propositionmain2} and \ref{propositionmain3}.
In short, by the following table 

\begin{center}
\begin{tabular}{|c|c|c|c|}
\hline
Entry & Information &Monomial & Its coefficient \\
\hline
$(1,2)$& $g_{12}^{(j)}=h_{12}^{(s)}=p_{12}^{(i,j)}=0$ & $m(f_{12}^{(j)})$ & $\alpha_j $\\
\hline
$(2,3)$& $h_{23}^{(s)}=p_{23}^{(i,j)}=0$ & $m(g_{23}^{(j)})$ & $ \pm \beta_j $\\
\hline
$(1,3)$&  & $w$ & $2 \gamma $\\
\hline
$(1,3)$&  & $m(p_{13}^{(i,j)})$ & $\beta_{i,j}$\\
\hline
\end{tabular}
\end{center}
where
\begin{eqnarray}
m(f_{12}^{(j)})&=& (z_{11}^1)^{n_1} \ldots (z_{11}^j)^{n_j-1} \ldots (z_{11}^s)^{n_s} y_{11}^1z_{12}^j ,\\
m(g_{23}^{(j)})&=& y_{11}^1z_{12}^j(z_{11}^1)^{n_1}\ldots (z_{11}^j)^{n_j-1}\ldots (z_{11}^s)^{n_s} ,\\
w&=& y_{13}^1(z_{11}^1)^{n_1}\ldots (z_{11}^s)^{n_s}, \\
m(p_{13}^{(i,j)})&=& (z_{11}^1)^{n_1}\ldots (z_{11}^i)^{n_i-1}\ldots (z_{11}^j)^{n_j-1}\ldots (z_{11}^s)^{n_s}y_{11}^1z_{12}^iz_{12}^j,
\end{eqnarray}
we have $\alpha_j=0, \ \beta_j=0, \ \gamma =0, \ \beta_{i,j}=0$ respectively. 
\end{proof}
\end{proposition}

\subsubsection{Subspace $B_{1N}$ where $n_s=1$}

By Observation \ref{observacaoordemdasvariaveis}, if $n_s=1$ then $n_1=\ldots =n_s=1$. In this case, 
$B_{1N}=\Gamma_{1s}$.

\begin{proposition} 
If $n_s =1$, then 
$\{f+Id\cap \Gamma_{1s}: \ f\in S_3\}$ is a basis for the vector space $\Gamma_{1s}(Id)$. In particular, $Id\cap \Gamma_{1s}=I\cap \Gamma_{1s}$.
\begin{proof}
If $f$ is $S_3$-standard then $f$ is $T_2$-standard in \cite[Definition 5.18]{vinkossca}. Now we can use the same proof 
of \cite[Lemma 6.5]{vinkossca}.
\end{proof}
\end{proposition}

\subsection{Subspaces $B_{MN}$ where $M \neq (0),(1)$ and $N\neq (0),(1)$}

Let $M=(m_1,\ldots , m_k)$ and $N=(n_1,\ldots ,n_s)$. In this section,
\[m= m_1+ \ldots + m_k \geq 2 \ \ \mbox{and} \ \  n=n_1+\ldots +n_s \geq 2.\]

\begin{definition}
Let $S_3\subset S_2$ be the set of all polynomials 
\[f^{({i_1})}, \ g^{({i_1})}, \ f^{(i,{i_1})}, \ g^{(i,{i_1})}, \   h^{(j_1,p_1)}  \in B_{MN}\] such that:
\begin{align}
& \bullet f^{({i_1})}=[z_{i_1},z_1,x_{i_2},\dots ,x_{i_{t-1}}] \in S_2,  \nonumber \\
& \bullet g^{({i_1})}=[y_{i_1},z_1,x_{i_2},\dots ,x_{i_{t-1}}] \in S_2, \nonumber \\
& \bullet f^{(i,{i_1})}=z_i[z_{i_1},x_{i_2},\dots ,x_{i_{t-1}}] \in S_2 \ \ \mbox{and} \ \  z_i\leq z_{i_1}   \label{polinomios},\\
& \bullet g^{(i,{i_1})}=z_i[y_{i_1},x_{i_2},\dots ,x_{i_{t-1}}]\in S_2, \nonumber \\
& \bullet h^{(j_1,p_1)}=[y_{j_1},x_{j_2},\dots ,x_{j_{t-2}}][y_{p_1},z_1]\in S_2  \nonumber, 
\end{align}
where $t=m+n$. If $f\in S_3$, we say that $f$ is an $S_3$-standard polynomial.
\end{definition}

\begin{proposition} \label{geradoresBMN}
The vector space $B_{MN}(I)$ is spanned by the set of all elements $f+I\cap B_{MN}$ where $f$ is $S_3$-standard.
\begin{proof}
This proposition is true for $\Gamma_{mn}(I)$. In fact, we can use the same proof as in \cite[Proposition 5.17]{vinkossca}. Now, if $x_i<x_j$, then $\varphi_{MN}(x_i)\leq \varphi_{MN}(x_j)$. Thus, by Proposition \ref{relacaodegammaeb}, the general case is proved.
\end{proof}
\end{proposition}

In $UT_3(F[y_{ij}^k,  z_{ij}^k])$ consider the \textit{sgeneric matrices} 
\[
Z_1=\left[
\begin{array}{ccc}
1 & 0 & 0\\
0 & 0 & 0\\
0 & 0 & -1
\end{array}
\right] \ \ , \ \ 
Z_l=\left[
\begin{array}{ccc}
1 & z_{12}^l & 0\\
0 & 0 & -z_{12}^l \\
0 & 0 & -1
\end{array}
\right] \ \ 
\mbox{and} \ \ Y_j=\left[
\begin{array}{ccc}
1 & y_{12}^j & y_{13}^j\\
0 & 0 & y_{12}^j\\
0 & 0 & 1
\end{array}
\right]\] 
for all $l\geq 2$ and $j\geq 1$. If $w(y_1,\ldots , y_k,z_1, \ldots , z_s)$ is $S_3$-standard then we write
\[w(Y_1,\ldots , Y_k,Z_1, \ldots , Z_s)=\sum_{a,b=1}^3 w_{ab} e_{ab}.\]

Since $\mathbb{F}$ is an infinite field we have the following lemma: 

\begin{lemma}\label{matrizgenericanova}
Let $Y_1,\ldots , Y_k,Z_1, \ldots , Z_s$ be sgeneric matrices. 
  If $f(y_1,\ldots , y_k,z_1, \ldots , z_s) \in Id$, then $f(Y_1,\ldots , Y_k,Z_1, \ldots , Z_s) =0$.
\end{lemma}

\begin{lemma} \label{lemadasigualdades}

Let $Z_l$ and $Y_l$ be the sgeneric matrices, where $l\geq 1$.  

a) If $m\geq 2$ is even then:
\begin{eqnarray}
 [Z_{i_1},Z_{i_2},\dots ,Z_{i_n},Y_{j_1},\dots ,Y_{j_m}]&=&(-1)^n(z_{12}^{i_2}-z_{12}^{i_1})(e_{12}-e_{23}), 
\ \mbox{where} \ n\geq 2; \nonumber
\end{eqnarray}

\vspace{-0.8cm}

\begin{eqnarray}
[Y_{j_1},Z_{i_1},\dots ,Z_{i_n},Y_{j_2},\dots ,Y_{j_m}]&=&(-1)^n(z_{12}^{i_1}-y_{12}^{j_1})(e_{12}-e_{23}),
\ \mbox{where} \ n\geq 2; \nonumber
\end{eqnarray}

\vspace{-0.8cm}

\begin{eqnarray}
Z_i[Z_{i_1},Z_{i_2},\dots ,Z_{i_{n-1}},Y_{j_1},\dots ,Y_{j_m}]&=&(-1)^{n-1}(z_{12}^{i_2}-z_{12}^{i_1})e_{12}+ \nonumber \\ 
&& (-1)^{n}z_{12}^i(z_{12}^{i_2}-z_{12}^{i_1})e_{13}, 
\ \mbox{where} \ n\geq 3; \nonumber
\end{eqnarray}

\vspace{-0.8cm}

\begin{eqnarray}
Z_i[Y_{j_1},Z_{i_1},\dots ,Z_{i_{n-1}},Y_{j_2},\dots ,Y_{j_m}]&=&(-1)^{n-1}(z_{12}^{i_1}-y_{12}^{j_1})e_{12} + \nonumber \\ 
&& (-1)^{n}z_{12}^i(z_{12}^{i_1}-y_{12}^{j_1})e_{13}, 
\ \mbox{where} \ n\geq 2; \nonumber
\end{eqnarray}

\vspace{-0.8cm}

\begin{eqnarray}
[Y_{j_1},Z_{i_1},\dots ,Z_{i_{n-1}},Y_{j_2},\dots ,Y_{j_{m-1}}][Y_{p_1},Z_{p_2}]&=&(-1)^n(z_{12}^{i_1}-y_{12}^{j_1})(z_{12}^{p_2}-y_{12}^{p_1})e_{13}, \nonumber \\ 
&&  
\ \mbox{where} \ n\geq 2. \nonumber
\end{eqnarray}

b) If $m\geq 2$ is odd then:
\begin{eqnarray}
 [Z_{i_1},Z_{i_2},\dots ,Z_{i_n},Y_{j_1},\dots ,Y_{j_m}]&=&(-1)^{n-1}(z_{12}^{i_2}-z_{12}^{i_1})(e_{12}+e_{23})+ \nonumber \\ 
&&2(-1)^{n}(z_{12}^{i_2}-z_{12}^{i_1})y_{12}^{j_m}e_{13},
  \ \mbox{where} \ n\geq 2; \nonumber
\end{eqnarray}

\vspace{-0.8cm}

\begin{eqnarray}
[Y_{j_1},Z_{i_1},\dots ,Z_{i_n},Y_{j_2},\dots ,Y_{j_m}]&=&(-1)^{n-1}(z_{12}^{i_1}-y_{12}^{j_1})(e_{12}+e_{23})+ \nonumber \\ 
&&2(-1)^{n}(z_{12}^{i_1}-y_{12}^{j_1})y_{12}^{j_m}e_{13},
  \ \mbox{where} \ n\geq 2; \nonumber
\end{eqnarray}

\vspace{-0.8cm}

\begin{eqnarray}
Z_i[Z_{i_1},Z_{i_2},\dots ,Z_{i_{n-1}},Y_{j_1},\dots ,Y_{j_m}]&=&(-1)^n(z_{12}^{i_2}-z_{12}^{i_1})e_{12} + \nonumber \\ 
&&(-1)^n(z_{12}^{i_2}-z_{12}^{i_1})(-2y_{12}^{j_m}+z_{12}^i)e_{13},
  \ \mbox{where} \ n\geq 3; \nonumber
\end{eqnarray}

\vspace{-0.8cm}

\begin{eqnarray}
Z_i[Y_{j_1},Z_{i_1},\dots ,Z_{i_{n-1}},Y_{j_2},\dots ,Y_{j_m}]&=&(-1)^n(z_{12}^{i_1}-y_{12}^{j_1})e_{12} + \nonumber \\ 
&&(-1)^n(z_{12}^{i_1}-y_{12}^{j_1})(-2y_{12}^{j_m}+z_{12}^i)e_{13},
  \ \mbox{where} \ n\geq 2; \nonumber
\end{eqnarray}

\vspace{-0.8cm}

\begin{eqnarray}
[Y_{j_1},Z_{i_1},\dots ,Z_{i_{n-1}},Y_{j_2},\dots ,Y_{j_{m-1}}][Y_{p_1},Z_{p_2}]&=&(-1)^{n-1}(z_{12}^{i_1}-y_{12}^{j_1})(z_{12}^{p_2}-y_{12}^{p_1})e_{13}, \nonumber \\ 
&&
  \ \mbox{where} \ n\geq 2. \nonumber
\end{eqnarray}
\end{lemma}

\begin{proof}
We leave the proof to the reader.
\end{proof}

\subsubsection{Case $m$ even and $n_1>1$}

Let $M=(m_1,\ldots , m_k)$,  $N=(n_1,\ldots ,n_s)$, $m= m_1+ \ldots + m_k \geq 2$  and $ n=n_1+\ldots +n_s \geq 2.$
In this subsection, we consider the case where $m$ is even and $n_1 >1$.

\begin{proposition} \label{meven}
If $m$ is even and $n_1>1$, then $\{f+Id\cap B_{MN}:f\in S_3\}$ is a basis for the vector space $B_{MN}(Id)$. In particular, $Id\cap B_{MN}=I\cap B_{MN}$.
\end{proposition}
\begin{proof}
Since $I\subset Id$, we have, by Proposition \ref{geradoresBMN},
\[ B_{MN}(Id)=\mbox{span}\{f+Id\cap B_{MN}:f\in S_3\}.\]

Consider some order $>$ on $L$ such that
\[ z_{12}^{i+1} > z_{12}^{i} > y_{12}^{i+1} > y_{12}^{i}\]
for all $i\geq 1$. 
Let $Z_l$ and $Y_l$ be the sgeneric matrices, where $l\geq 1$.  By Lemma \ref{lemadasigualdades} we have 
\begin{align*}
& [Z_{i_1},Z_1,\dots ,Y_{j_m}]=(-1)^{n+1}z_{12}^{i_1}(e_{12}-e_{23}), \\
& [Y_{j_1},Z_1,\dots ,Y_{j_m}]=(-1)^{n+1}y_{12}^{j_1}(e_{12}-e_{23}), \\
& Z_i[Z_{i_1},Z_1,\dots ,Y_{j_m}]=(-1)^nz_{12}^{i_1}e_{12}-(-1)^nz_{12}^iz_{12}^{i_1}e_{13}, \\
& Z_i[Y_{j_1},Z_1,\dots ,Y_{j_m}]=(-1)^ny_{12}^{j_1}e_{12}-(-1)^nz_{12}^iy_{12}^{j_1}e_{13},\\
& [Y_{j_1},Z_1,\dots ,Y_{j_{m-1}}][Y_{p_1},Z_1]=(-1)^ny_{12}^{j_1}y_{12}^{p_1}e_{13}.
\end{align*}

Let $f^{(i_1)},g^{(j_1)},f^{(i,i_1)}, g^{(i,j_1)},h^{(j_1,p_1)}$ as in (\ref{polinomios})
and suppose
\[ \sum \alpha_{i_1}f^{(i_1)}+\sum \beta_{j_1}g^{(j_1)}+\sum \alpha_{i,i_1}f^{(i,i_1)}+\sum \beta_{i,j_1}g^{(i,j_1)}+\sum \gamma_{j_1,p_1}h^{(j_1,p_1)} \in Id, \]
where $\alpha_{i_1},\beta_{j_1}, \alpha_{i,i_1},\beta_{i,j_1},\gamma_{j_1,p_1} \in \mathbb{F}$.
Now we use the same arguments as in the previous propositions. 
In short, by the following table 

\begin{center}
\begin{tabular}{|c|c|c|c|}
\hline
Entry & Information &Monomial & Its coefficient \\
\hline
$(2,3)$& $f_{23}^{(i,i_1)}=g_{23}^{(i,j_1)}=h_{23}^{(j_1,p_1)}=0$ & $m(f_{23}^{(i_1)})$ & $\pm \alpha_{i_1}$\\
\hline
$(2,3)$& $f_{23}^{(i,i_1)}=g_{23}^{(i,j_1)}=h_{23}^{(j_1,p_1)}=0$ & $m(g_{23}^{(j_1)})$ & $ \pm \beta_{j_1} $\\
\hline
$(1,3)$& $i>1$ & $m(f_{13}^{(i,i_1)}) $ & $\pm \alpha_{i,i_1} $\\
\hline
$(1,3)$& $i>1$ & $m(g_{13}^{(i,j_1)}) $ & $\pm \beta_{i,j_1} $\\
\hline
$(1,2)$&    $i=1$   & $m(f_{12}^{(1,i_1)}) $ & $\pm \alpha_{1,i_1} $\\
\hline
$(1,2)$&    $i=1$   & $m(g_{12}^{(1,j_1)}) $ & $\pm \beta_{1,j_1} $\\
\hline
$(1,3)$&       & $m(h_{13}^{(j_1,p_1)})$ & $\pm \gamma_{j_1,p_1}$\\
\hline
\end{tabular}
\end{center}
where
\begin{center}
\begin{tabular}{ll}
\vspace{0.2cm}

$m(f_{23}^{(i_1)})= z_{12}^{i_1}$ , & \ \ \ \ \ $ m(g_{23}^{(j_1)})=y_{12}^{j_1}$,\\
\vspace{0.2cm}

$m(f_{13}^{(i,i_1)})= z_{12}^iz_{12}^{i_1}$, & \ \ \ \ \ $m(g_{13}^{(i,j_1)})=z_{12}^iy_{12}^{j_1}$,\\
\vspace{0.2cm}

$m(f_{12}^{(1,i_1)})=z_{12}^{i_1} $, & \ \ \ \ \ $m(g_{12}^{(1,j_1)})=y_{12}^{j_1}$,\\
$m(h_{13}^{(j_1,p_1)})=y_{12}^{j_1}y_{12}^{p_1}$, & 
\end{tabular}
\end{center}
we have $\alpha_{i_1}=0$, $\beta_{j_1}=0$, $\alpha_{i,i_1}=0$, $\beta_{i,j_1}=0$, $\alpha_{1,i_1}=0$, $\beta_{1,j_1}=0$, $\gamma_{j_1,p_1}=0$, respectively.
\end{proof}

\subsubsection{Case $m$ even and $n_1=1$}

Let $M=(m_1,\ldots , m_k)$,  $N=(n_1,\ldots ,n_s)$, $m= m_1+ \ldots + m_k \geq 2$  and $ n=n_1+\ldots +n_s \geq 2.$
In this subsection, we consider the case where $m$ is even and $n_1 =1$.

\begin{proposition}
If $m$ is even and $n_1=1$, then $\{f+Id\cap B_{MN}:f\in S_3\}$ is a basis for the vector space $B_{MN}(Id)$. In particular, $Id\cap B_{MN}=I\cap B_{MN}$.
\begin{proof}
Since $I\subset Id$, we have, by Proposition \ref{geradoresBMN},
\[ B_{MN}(Id)=\mbox{span}\{f+Id\cap B_{MN}:f\in S_3\}.\]

Consider some order $>$ on $L$ such that
\[ y_{12}^{i+1} > y_{12}^{i} >      z_{12}^{i+1} > z_{12}^{i} \]
for all $i\geq 1$. Let $Z_l$ and $Y_l$ be the sgeneric matrices, where $l\geq 1$. By Lemma \ref{lemadasigualdades} we have
\begin{align*}
& [Z_{i_1},Z_1,\dots , Y_{j_m}]=(-1)^{n+1}z_{12}^{i_1}(e_{12}-e_{23}), \\
& [Y_{j_1},Z_1,\dots ,Y_{j_m}]=(-1)^{n+1}y_{12}^{j_1}(e_{12}-e_{23}), \\
& Z_i[Z_{i_1},Z_1,\dots ,Y_{j_m}]=(-1)^nz_{12}^{i_1}e_{12}-(-1)^nz_{12}^iz_{12}^{i_1}e_{13}, \\
& Z_1[Z_{i_1},Z_2,\dots ,Y_{j_m}]=(-1)^{n-1}(z_{12}^2-z_{12}^{i_1})e_{12}, \\
& Z_i[Y_{j_1},Z_1,\dots ,Y_{j_m}]=(-1)^ny_{12}^{j_1}e_{12}-(-1)^nz_{12}^iy_{12}^{j_1}e_{13}, \\
& Z_1[Y_{j_1},Z_2,\dots ,Y_{j_m}]=(-1)^{n-1}(z_{12}^2-y_{12}^{j_1})e_{12}, \\
& [Y_{j_1},Z_2,\dots ,Y_{j_{m-1}}][Y_{p_1},Z_1]=(-1)^{n-1}(z_{12}^2-y_{12}^{j_1})y_{12}^{p_1}e_{13}.
\end{align*}

Let $f^{(i_1)},g^{(j_1)},f^{(i,i_1)},g^{(i,j_1)}
,h^{(j_1,p_1)}$ as in (\ref{polinomios}) and suppose
\[ \sum \alpha_{i_1}f^{(i_1)}+\sum \beta_{j_1}g^{(j_1)}+\sum \alpha_{i,i_1}f^{(i,i_1)}+\sum \beta_{i,j_1}g^{(i,j_1)}+\sum \gamma_{j_1,p_1}h^{(j_1,p_1)} \in Id, \]
where $\alpha_{i_1},\beta_{j_1}, \alpha_{i,i_1},\beta_{i,j_1},\gamma_{j_1,p_1} \in \mathbb{F}$.
Now we use the same arguments as in the previous propositions. 
In short, by the following table 

\begin{center}
\begin{tabular}{|c|c|c|c|}
\hline
Entry & Information &Monomial & Its coefficient \\
\hline
$(2,3)$& $f_{23}^{(i,i_1)}=g_{23}^{(i,j_1)}=h_{23}^{(j_1,p_1)}=0$ & $m(g_{23}^{(j_1)})$ & $\pm \beta_{j_1}$\\
\hline
$(2,3)$& $f_{23}^{(i,i_1)}=g_{23}^{(i,j_1)}=h_{23}^{(j_1,p_1)}=0$ & $m(f_{23}^{(i_1)})$ & $ \pm \alpha_{i_1} $\\
\hline
$(1,3)$&  & $m(h_{13}^{(j_1,p_1)}) $ & $\pm \gamma_{j_1,p_1} $\\
\hline
$(1,3)$& $i>1$ & $m(g_{13}^{(i,j_1)}) $ & $\pm \beta_{i,j_1} $\\
\hline
$(1,3)$& $i>1$ & $m(f_{13}^{(i,i_1)})$ & $\pm \alpha_{i,i_1}$\\
\hline
$(1,2)$& $i=1$ & $m(g_{12}^{(1,j_1)}) $ & $\pm \beta_{1,j_1} $\\
\hline
$(1,2)$& $i=1$ & $m(f_{12}^{(1,i_1)}) $ & $\pm \alpha_{1,i_1} $\\
\hline
\end{tabular}
\end{center}
where
\begin{center}
\begin{tabular}{ll}
\vspace{0.2cm}

$ m(g_{23}^{(j_1)})=y_{12}^{j_1}$ , & \ \ \ \ \ $m(f_{23}^{(i_1)})= z_{12}^{i_1}$,\\
\vspace{0.2cm}

$m(h_{13}^{(j_1,p_1)})=y_{12}^{j_1}y_{12}^{p_1}$, & \ \ \ \ \ $m(g_{13}^{(i,j_1)})=z_{12}^iy_{12}^{j_1}$,\\
\vspace{0.2cm}

$m(f_{13}^{(i,i_1)})= z_{12}^iz_{12}^{i_1}$, & \ \ \ \ \ $m(g_{12}^{(1,j_1)})=y_{12}^{j_1}$,\\

$m(f_{12}^{(1,i_1)})=z_{12}^{i_1} $, &
\end{tabular}
\end{center}
we have $\beta_{j_1}=0$, $\alpha_{i_1}=0$, $\gamma_{j_1,p_1}=0$ , $\beta_{i,j_1}=0$, $\alpha_{i,i_1}=0$, 
$\beta_{1,j_1}=0$, $\alpha_{1,i_1}=0$, respectively.
\end{proof}
\end{proposition}

\subsubsection{Case $m$ odd and $n_1>1$}

Let $M=(m_1,\ldots , m_k)$,  $N=(n_1,\ldots ,n_s)$, $m= m_1+ \ldots + m_k \geq 2$  and $ n=n_1+\ldots +n_s \geq 2.$
In this subsection, we consider the case where $m$ is odd and $n_1 >1$.

\begin{proposition}\label{polinomioslimimpar}
If $m$ is odd and $n_1>1$, then $\{f+Id\cap B_{MN}:f\in S_3\}$ is a basis for the vector space $B_{MN}(Id)$. In particular, $Id\cap B_{MN}=I\cap B_{MN}$.
\begin{proof}
Let $Z_l$ and $Y_l$ be the sgeneric matrices, where $l\geq 1$. By Lemma \ref{lemadasigualdades} we have
\begin{align*}
& [Z_{i_1},Z_1,\dots ,Y_{j_m}]=(-1)^nz_{12}^{i_1}(e_{12}+e_{23})-2(-1)^nz_{12}^{i_1}y_{12}^{j_m}e_{13}, \\
& [Y_{j_1},Z_1,\dots ,Y_{j_m}]=(-1)^ny_{12}^{j_1}(e_{12}+e_{23})-2(-1)^ny_{12}^{j_1}y_{12}^{j_m}e_{13}, \\
& Z_i[Z_{i_1},Z_1,\dots ,Y_{j_m}]=(-1)^{n+1}z_{12}^{i_1}e_{12}+(-1)^{n+1}z_{12}^{i_1}(-2y_{12}^{j_m}+z_{12}^i)e_{13}, \\
& Z_i[Y_{j_1},Z_1,\dots ,Y_{j_m}]=(-1)^{n+1}y_{12}^{j_1}e_{12}+(-1)^{n+1}y_{12}^{j_1}(-2y_{12}^{j_m}+z_{12}^i)e_{13}, \\
& [Y_{j_1},Z_1,\dots ,Y_{j_{m-1}}][Y_{p_1},Z_1]=(-1)^{n+1}y_{12}^{j_1}y_{12}^{p_1}e_{13}.
\end{align*}
Now we use the same order $>$, table and leading monomials in Proposition \ref{meven}. 
\end{proof}
\end{proposition}

\subsubsection{Case $m$ odd where $n_1=1$ and $m_k>1$}

Let $M=(m_1,\ldots , m_k)$,  $N=(n_1,\ldots ,n_s)$, $m= m_1+ \ldots + m_k \geq 2$  and $ n=n_1+\ldots +n_s \geq 2.$
In this subsection, we consider the case $m$ odd where $n_1=1$ and $m_k>1$

\begin{proposition}
If $m$ is odd, $n_1=1$ and $m_k>1$, then $\{f+Id\cap B_{MN}:f\in S_3\}$ is a basis for the vector space $B_{MN}(Id)$. In particular, $Id\cap B_{MN}=I\cap B_{MN}$.
\begin{proof}
Since $I\subset Id$, we have, by Proposition \ref{geradoresBMN},
\[ B_{MN}(Id)=\mbox{span}\{f+Id\cap B_{MN}:f\in S_3\}.\]

Consider some order $>$ on $L$ such that
\[ z_{12}^{i+1} > z_{12}^{i} > y_{12}^{i+1} > y_{12}^{i} \]
for all $i\geq 1$. Let $Z_l$ and $Y_l$ be the sgeneric matrices, where $l\geq 1$. By Lemma \ref{lemadasigualdades} we have
\begin{align*}
& [Z_{i_1},Z_1,\dots ,Y_k]=(-1)^nz_{12}^{i_1}(e_{12}+e_{23})-2(-1)^nz_{12}^{i_1}y_{12}^ke_{13}, \\
& [Y_{j_1},Z_1,\dots ,Y_k]=(-1)^ny_{12}^{j_1}(e_{12}+e_{23})-2(-1)^ny_{12}^{j_1}y_{12}^ke_{13}, \\
& Z_i[Z_{i_1},Z_1,\dots ,Y_k]=(-1)^{n+1}z_{12}^{i_1}e_{12}+(-1)^{n+1}z_{12}^{i_1}(-2y_{12}^k+z_{12}^i)e_{13}, \\
& Z_1[Z_{i_1},Z_2,\dots ,Y_k]=(-1)^n(z_{12}^2-z_{12}^{i_1})e_{12}-2(-1)^n(z_{12}^2-z_{12}^{i_1})y_{12}^ke_{13}, \\
& Z_i[Y_{j_1},Z_1,\dots ,Y_k]=(-1)^{n+1}y_{12}^{j_1}e_{12}+(-1)^{n+1}y_{12}^{j_1}(-2y_{12}^k+z_{12}^i)e_{13}, \\
& Z_1[Y_{j_1},Z_2,\dots ,Y_k]=(-1)^n(z_{12}^2-y_{12}^{j_1})e_{12}-2(-1)^n(z_{12}^2-y_{12}^{j_1})y_{12}^ke_{13}, \\
& [Y_{j_1},Z_2,\dots ,Y_{j_{m-1}}][Y_{p_1},Z_1]=(-1)^n(z_{12}^2-y_{12}^{j_1})y_{12}^{p_1}e_{13}. 
\end{align*}

Let $f^{(i_1)},g^{(j_1)},f^{(i,i_1)},g^{(i,j_1)},h^{(j_1,p_1)}$ as in (\ref{polinomios}), and suppose
\[ \sum \alpha_{i_1}f^{(i_1)}+\sum \beta_{j_1}g^{(j_1)}+\sum \alpha_{i,i_1}f^{(i,i_1)}+\sum \beta_{i,j_1}g^{(i,j_1)}+\sum \gamma_{j_1,p_1}h^{(j_1,p_1)} \in Id, \]
where $\alpha_{i_1},\beta_{j_1}, \alpha_{i,i_1},\beta_{i,j_1},\gamma_{j_1,p_1} \in \mathbb{F}$. 
By the following table 

\begin{center}
\begin{tabular}{|c|c|c|c|}
\hline
Entry & Information &Monomial & Its coefficient \\
\hline
$(2,3)$& $f_{23}^{(i,i_1)}=g_{23}^{(i,j_1)}=h_{23}^{(j_1,p_1)}=0$ & $m(f_{23}^{(i_1)})$ & $\pm \alpha_{i_1}$\\
\hline
$(2,3)$& $f_{23}^{(i,i_1)}=g_{23}^{(i,j_1)}=h_{23}^{(j_1,p_1)}=0$ & $m(g_{23}^{(j_1)})$ & $ \pm \beta_{j_1} $\\
\hline
$(1,3)$& $i>1$ & $m(f_{13}^{(i,i_1)}) $ & $\pm \alpha_{i,i_1} $\\
\hline
$(1,2)$& $i=1$ & $m(f_{12}^{(1,i_1)}) $ & $\pm \alpha_{1,i_1} $\\
\hline
$(1,3)$& $i>2$ & $m(g_{13}^{(i,j_1)})$ & $\pm \beta_{i,j_1}$\\
\hline
$(1,3)$& $j_1<k$ & $w$ & $\pm \gamma_{j_1,p_1} $\\
\hline
\end{tabular}
\end{center}
where
\begin{center}
\begin{tabular}{ll}
\vspace{0.2cm}

$m(f_{23}^{(i_1)})=z_{12}^{i_1}$, & \ \ \ \ \ $m(g_{23}^{(j_1)})=y_{12}^{j_1}$, \\
\vspace{0.2cm}

$m(f_{13}^{(i,i_1)})=z_{12}^{i_1}z_{12}^i$, & \ \ \ \ \ $m(f_{12}^{(1,i_1)})=z_{12}^{i_1}$, \\
$m(g_{13}^{(i,j_1)})=y_{12}^{j_1}z_{12}^i$, & \ \ \ \ \ $w=y_{12}^{j_1}y_{12}^{p_1}$,
\end{tabular}
\end{center}
we have $\alpha_{i_1}=0$, $\beta_{j_1}=0$, $\alpha_{i,i_1}=0$ for $i>1$, $\alpha_{1,i_1}=0$, $\beta_{i,j_1}=0$ for $i>2$ and $\gamma_{j_1,p_1}=0$ for $j_1<k$, respectively.
Thus, now we have
\[v= \sum \beta_{2,j_1}g^{(2,j_1)}+\sum \beta_{1,j_1}g^{(1,j_1)}+\sum \gamma_{k,p_1}h^{(k,p_1)}\in Id.\]
The coefficient of $y_{12}^{j_1}$ in $v_{12}$ is 
\[ \beta_{1,j_1}+\beta_{2,j_1}=0\]
for all $j_1=1,\dots ,k$; and the coefficient of $y_{12}^ly_{12}^k$ in $v_{13}$  is 
\[ -2\beta_{1,l}-2\beta_{2,l}+\gamma_{k,l}=0\]
for all $l=1,\dots ,k$. Therefore, $\gamma_{k,l}=0$ for all $l=1,\dots ,k$.

For the remaining coefficients, by the following table 

\begin{center}
\begin{tabular}{|c|c|c|c|}
\hline
Entry & Information &Monomial & Its coefficient \\
\hline
$(1,3)$ & $j_1<k$ & $m(g_{13}^{(2,j_1)})$ & $\pm \beta_{2,j_1}$ \\
\hline
$(1,3)$ & $j_1<k$ & $u$ & $\pm 2\beta_{1,j_1}$ \\
\hline 
$(1,2)$ & $j_1=k$ & $m(g_{12}^{(1,k)})$ & $\pm \beta_{1,k}$ \\
\hline
$(1,2)$ & $j_1=k$ & $m(g_{12}^{(2,k)})$ & $\pm \beta_{2,k}$ \\
\hline
\end{tabular}
\end{center}
where
\begin{center}
\begin{tabular}{ll}
\vspace{0.2cm}

$m(g_{13}^{(2,j_1)})=y_{12}^{j_1}z_{12}^2$, & \ \ \ \ \ $u=y_{12}^{j_1}y_{12}^k$, \\
$m(g_{12}^{(1,k)})=z_{12}^2$, & \ \ \ \ \ $m(g_{12}^{(2,k)})=y_{12}^k$,
\end{tabular}
\end{center}
we have $\beta_{2,j_1}=0$, $\beta_{1,j_1}=0$, $\beta_{1,k}=0$ and $\beta_{2,k}=0$, respectively.
\end{proof}
\end{proposition}

\subsubsection{Case $m$ odd where $n_1=m_k=1$ and char($\mathbb{F}$)$>3$}

Let $M=(m_1,\ldots , m_k)$,  $N=(n_1,\ldots ,n_s)$, $m= m_1+ \ldots + m_k \geq 2$  and $ n=n_1+\ldots +n_s \geq 2.$
In this subsection, we consider the case $m$ odd where $n_1=m_k=1$ and char($\mathbb{F}$)$>3$.

\begin{proposition}
If $m$ is odd, $n_1=m_k=1$ and $\mbox{char}(\mathbb{F})>3$, then $Id\cap B_{MN}=I\cap B_{MN}$.
\begin{proof}
By Observation \ref{observacaoordemdasvariaveis} we have $m_1=\ldots =m_{k-1}=m_k=1$. 

If 
$n_s=1$ then $n_1=n_2=\ldots =n_s=1$ and we can use the same proof of \cite[Lemma 6.4]{vinkossca}.

Suppose $n_s >1$. By a change of variables $z_1 \longleftrightarrow z_s$ we can suppose $n_1 >1$. Note that
\[n_s \leq n_2 \leq n_3 \leq \ldots n_{s-1} \leq n_1.\]
But the Proposition \ref{polinomioslimimpar} is also true  in this case, the proof is the same. 
\end{proof}
\end{proposition}

\subsubsection{Case $m$ odd where $n_1=m_k=1$ and char($\mathbb{F}$)$=3$}

Let $M=(m_1,\ldots , m_k)$,  $N=(n_1,\ldots ,n_s)$, $m= m_1+ \ldots + m_k \geq 2$  and $ n=n_1+\ldots +n_s \geq 2.$
In this subsection, we consider the case $m$ odd where $n_1=m_k=1$ and char($\mathbb{F}$)$=3$.

We remember that $S_3$ is the set of all polynomials defined in (\ref{polinomios}).

\begin{definition}
Denote by $S_4$ the set
 \[S_4=S_3-\{g^{(1,k)}\}.\]
 We say that the polynomials in $S_4$ are $S_4$-standard.
\end{definition}

\begin{proposition} \label{lemaultimo}
The vector space $B_{MN}(I)$ is spanned by the set of all elements $f+I\cap B_{MN}$ where $f\in B_{MN}$ is $S_4$-standard.
\begin{proof}
We work modulo $I$. By Proposition \ref{geradoresBMN}, it is sufficient to prove that $g^{(1,k)}$ is a linear combination of $S_4$-standard polynomials.
In fact, we will prove that

\begin{equation}\label{igualdadedosgs}
g^{(1,k)}=g^{(1,k-1)}-g^{(2,k-1)}+g^{(2,k)}.
\end{equation} 

By Lemma \ref{relacoes5} - (b,c), we have
\[z_{l_n} \ldots z_{l_3} z_1[y_k,z_{l_2},y_1,\dots ,y_{k-1}]= (-1)^n z_1[y_k,z_{l_2},z_{l_3},\ldots , z_{l_n}, y_1,\dots ,y_{k-1}] .\]
Thus it is sufficient 
to prove (\ref{igualdadedosgs}) when $n=2$ that is 
\begin{align*}
z_1[y_k,z_2,y_1,\dots ,y_{k-1}]= & z_1[y_{k-1},z_2,y_1,\dots ,y_k] -z_2[y_{k-1},z_1,y_1,\dots ,y_k]\\   & +z_2[y_k,z_1,y_1,\dots ,y_{k-1}].
\end{align*}

\noindent {\bf {Claim:}} If $i\neq j$ and $a\neq b$, then:
\begin{align*}
z_i[y_1,z_j,y_a,\dots ,y_b]= & z_i[y_b,z_j,y_1,\dots ,y_a]-z_i[y_b,y_1,z_j,\dots ,y_a] \\
& +2z_i[y_1,z_j,\dots ][y_a,y_b]. 
\end{align*}

In fact, by Lemma \ref{relacoes2}-c), equality $[[a,b],c,d]-[[a,b],d,c]=[[a,b],[c,d]]$, Jacobi identity and Proposition \ref{ident}-b), we obtain
\begin{align*}
z_i[y_1,z_j,y_a,\dots ,y_b] & =z_i[y_1,z_j,\dots ,y_a,y_b] \\
 & =z_i[y_1,z_j,\dots ,y_b,y_a]+z_i[[y_1,z_j,\dots ],[y_a,y_b]] \\
 & =z_i[y_1,z_j,y_b,\dots ,y_a]+2z_i[y_1,z_j,\dots ][y_a,y_b] \\
 & =z_i[y_b,z_j,y_1,\dots ,y_a]-z_i[y_b,y_1,z_j,\dots ,y_a] \\
 & \ \ \ +2z_i[y_1,z_j,\dots ][y_a,y_b],
\end{align*} 
and the claim is proved. 

Now, by the Jacobi identity, we have
\[ g^{(1,k)}=z_1[y_1,z_2,y_k,\dots ,y_{k-1}]-z_1[y_1,y_k,z_2,\dots ,y_{k-1}]\]
and applying Lemma \ref{relacoes4} and Jacobi identity in the second summand, 
\begin{align*}
g^{(1,k)} & = z_1[y_1,z_2,y_k,\dots ,y_{k-1}]-z_2[y_1,y_k,z_1,\dots ,y_{k-1}] \\
 & \ \ \ +[y_1,y_k][z_2,z_1,\dots ,y_{k-1}] \\
 & = z_1[y_1,z_2,y_k,\dots ,y_{k-1}]+g^{(2,k)} \\
 & \ \ \ -z_2[y_1,z_1,y_k,\dots ,y_{k-1}]+[y_1,y_k][z_2,z_1,\dots ,y_{k-1}].
\end{align*}
By applying the Claim in the summands $z_1[y_1,z_2,y_k,\dots ,y_{k-1}]$ and $z_2[y_1,z_1,y_k,\dots ,y_{k-1}]$, we have
\[g^{(1,k)}  = g^{(1,k-1)}-g^{(2,k-1)}+g^{(2,k)} +f\] 
where
\begin{align}  
 f=    & \ \ -z_1[y_{k-1},y_1,z_2,\dots ,y_k]+z_2[y_{k-1},y_1,z_1,\dots ,y_k]  \nonumber \\ 
          & \ \ -2z_2[y_1,z_1,\dots ][y_k,y_{k-1}]+2z_1[y_1,z_2,\dots ][y_k,y_{k-1}] \nonumber \\
          & \ \ +[y_1,y_k][z_2,z_1,\dots ,y_{k-1}]. \nonumber
\end{align}
We shall prove that $f=0$. By Lemma \ref{relacoes4} and Lemma \ref{relacoes1}-iii),
\begin{align*}
-z_1[y_{k-1},y_1,z_2,\dots ,y_k]+z_2[y_{k-1},y_1,z_1,\dots ,y_k] & = [y_{k-1},y_1][z_2,z_1,\dots ,y_k] \\
 & = - [y_{k-1},y_1,y_k][z_2,z_1,\dots ].
\end{align*}
By Lemma \ref{relacoes1}-iii),
\[ [y_1,y_k][z_2,z_1,\dots ,y_{k-1}]=-[y_1,y_k,y_{k-1}][z_2,z_1,\dots], \]
and then, applying the Jacobi identity, Lemma \ref{relacoes1}-iii), Lemma \ref{relacoes2}-b) and Proposition \ref{ident}-ii) we have:
\begin{align*}
 & - [y_{k-1},y_1,y_k][z_2,z_1,\dots ]-[y_1,y_k,y_{k-1}][z_2,z_1,\dots] \\
 & = [y_k,y_{k-1},y_1][z_2,z_1,\dots ]\\
 & = -[y_k,y_{k-1}][z_2,z_1,y_1,\dots ] \\
 & = [z_2,z_1,y_1,\dots ][y_k,y_{k-1}] \\
 & = [y_1,z_1,z_2,\dots ][y_k,y_{k-1}]-[y_1,z_2,z_1,\dots ][y_k,y_{k-1}].
\end{align*}

By Proposition \ref{ident}-v) and Lemma \ref{relacoes2}-b),
\begin{align*}
& -2z_2[y_1,z_1,\dots ][y_k,y_{k-1}]+2z_1[y_1,z_2,\dots ][y_k,y_{k-1}] = \\
& 2[y_1,z_1,z_2,\dots ][y_k,y_{k-1}]-2[y_1,z_2,z_1\dots ][y_k,y_{k-1}].
\end{align*}

Therefore, since $\mbox{char}(\mathbb{F})=3$, we have
\begin{align*}
f & = [y_1,z_1,z_2,\dots ][y_k,y_{k-1}]-[y_1,z_2,z_1,\dots ][y_k,y_{k-1}] \\
 & \ \ +2[y_1,z_1,z_2,\dots ][y_k,y_{k-1}]-2[y_1,z_2,z_1\dots ][y_k,y_{k-1}] \\
 & = 3[y_1,z_1,z_2,\dots ][y_k,y_{k-1}]-3[y_1,z_2,z_1\dots ][y_k,y_{k-1}] \\
 & = 0.
\end{align*}
We finished the proof.
\end{proof}
\end{proposition}

\begin{proposition}
If $m$ is odd, $n_1=m_k=1$ and $\mbox{char}(\mathbb{F})=3$, then $\{f+Id\cap B_{MN}:f\in S_4\}$ is a basis for the vector space $B_{MN}(Id)$. In particular, $Id\cap B_{MN}=I\cap B_{MN}$.
\begin{proof}
Since $I\subset Id$, we have, by Proposition \ref{lemaultimo},
\[ B_{MN}(Id)=\mbox{span}\{f+Id\cap B_{MN}:f\in S_4\}.\]

Consider some order $>$ on $L$ such that
\[ z_{12}^{i+1} > z_{12}^{i} > y_{12}^{i+1} > y_{12}^{i} \]
for all $i\geq 1$. Let $Z_l$ and $Y_l$ be the sgeneric matrices, where $l\geq 1$. By Lemma \ref{lemadasigualdades} we have
\begin{align*}
& [Z_{i_1},Z_1,\dots ,Y_k]=(-1)^nz_{12}^{i_1}(e_{12}+e_{23})-2(-1)^nz_{12}^{i_1}y_{12}^ke_{13}, \\
& [Y_{j_1},Z_1,\dots ,Y_{j_m}]=(-1)^ny_{12}^{j_1}(e_{12}+e_{23})-2(-1)^ny_{12}^{j_1}y_{12}^{j_m}e_{13}, \\
& Z_i[Z_{i_1},Z_1,\dots ,Y_k]=(-1)^{n+1}z_{12}^{i_1}e_{12}+(-1)^{n+1}z_{12}^{i_1}(-2y_{12}^k+z_{12}^i)e_{13}, \\
& Z_1[Z_{i_1},Z_2,\dots ,Y_k]=(-1)^n(z_{12}^2-z_{12}^{i_1})e_{12}-2(-1)^n(z_{12}^2-z_{12}^{i_1})y_{12}^ke_{13}, \\
& Z_i[Y_{j_1},Z_1,\dots ,Y_k]=(-1)^{n+1}y_{12}^{j_1}e_{12}+(-1)^{n+1}y_{12}^{j_1}(-2y_{12}^k+z_{12}^i)e_{13}, \\
& Z_i[Y_k,Z_1,\dots ,Y_{k-1}]=(-1)^{n+1}y_{12}^ke_{12}+(-1)^{n+1}y_{12}^k(-2y_{12}^{k-1}+z_{12}^i)e_{13}, \\
& Z_1[Y_{j_1},Z_2,\dots ,Y_k]=(-1)^n(z_{12}^2-y_{12}^{j_1})e_{12}-2(-1)^n(z_{12}^2-y_{12}^{j_1})y_{12}^ke_{13}, \\
& [Y_{j_1},Z_2,\dots ,Y_{j_{m-1}}][Y_{p_1},Z_1]=(-1)^n(z_{12}^2-y_{12}^{j_1})y_{12}^{p_1}e_{13}.
\end{align*}

Let $f^{(i_1)},g^{(j_1)},f^{(i,i_1)},g^{(i,j_1)},h^{(j_1,p_1)}$ be $S_4$-standard polynomials, and suppose
\[ \sum \alpha_{i_1}f^{(i_1)}+\sum \beta_{j_1}g^{(j_1)}+\sum \alpha_{i,i_1}f^{(i,i_1)}+\sum \beta_{i,j_1}g^{(i,j_1)}+\sum \gamma_{j_1,p_1}h^{(j_1,p_1)} \in Id, \]
where $\alpha_{i_1},\beta_{j_1}, \alpha_{i,i_1},\beta_{i,j_1},\gamma_{j_1,p_1} \in \mathbb{F}$. Now we use the same arguments as in the previous propositions.
In short, by the following table 

\begin{center}
\begin{tabular}{|c|c|c|c|}
\hline
Entry & Information &Monomial & Its coefficient \\
\hline
$(2,3)$& $f_{23}^{(i,i_1)}=g_{23}^{(i,j_1)}=h_{23}^{(j_1,p_1)}=0$ & $m(f_{23}^{(i_1)})$ & $\pm \alpha_{i_1}$\\
\hline
$(2,3)$& $f_{23}^{(i,i_1)}=g_{23}^{(i,j_1)}=h_{23}^{(j_1,p_1)}=0$ & $m(g_{23}^{(j_1)})$ & $ \pm \beta_{j_1} $\\
\hline
$(1,3)$& $i>1$ & $m(f_{13}^{(i,i_1)}) $ & $\pm \alpha_{i,i_1} $\\
\hline
$(1,2)$& $i=1$ & $m(f_{12}^{(1,i_1)}) $ & $\pm \alpha_{1,i_1} $\\
\hline
$(1,3)$& $i>2$ & $m(g_{13}^{(i,j_1)})$ & $\pm \beta_{i,j_1}$\\
\hline
$(1,3)$& $j_1<k$ & $w$ & $\pm \gamma_{j_1,p_1} $\\
\hline
$(1,2)$& $i=2$ & $m(g_{12}^{(2,k)})$ & $\pm \beta_{2,k}$ \\
\hline
\end{tabular}
\end{center}
where
\begin{center}
\begin{tabular}{ll}
\vspace{0.2cm}

$m(f_{23}^{(i_1)})=z_{12}^{i_1}$, & \ \ \ \ \ $m(g_{23}^{(j_1)})=y_{12}^{j_1}$, \\
\vspace{0.2cm}

$m(f_{13}^{(i,i_1)})=z_{12}^{i_1}z_{12}^i$, & \ \ \ \ \ $m(f_{12}^{(1,i_1)})=z_{12}^{i_1}$, \\
\vspace{0.2cm}

$m(g_{13}^{(i,j_1)})=y_{12}^{j_1}z_{12}^i$, & \ \ \ \ \ $w=y_{12}^{j_1}y_{12}^{p_1}$, \\
$m(g_{12}^{(2,k)})=y_{12}^k$, &
\end{tabular}
\end{center}
we have $\alpha_{i_1}=0$, $\beta_{j_1}=0$, $\alpha_{i,i_1}=0$ for $i>1$, $\alpha_{1,i_1}=0$, $\beta_{i,j_1}=0$ for $i>2$ ,
 $\gamma_{j_1,p_1}=0$ for $j_1<k$ and $\beta_{2,k}=0$, respectively.

Thus, now we have
\[ \sum_{j_1=1}^{k-1} \beta_{2,j_1}g^{(2,j_1)}+\sum_{j_1=1}^{k-1} \beta_{1,j_1}g^{(1,j_1)}+\sum_{p_1=1}^{k-1} \gamma_{k,p_1}h^{(k,p_1)}\in Id.\]

By the monomial $y_{12}^{j_1}$ in the $(1,2)$-entry, we have
\[ \beta_{1,j_1}+\beta_{2,j_1}=0\]
for all $j_1=1,\dots ,k-1$, and by the monomial $y_{12}^ly_{12}^k$ in the $(1,3)$-entry  we have
\[ -2\beta_{1,l}-2\beta_{2,l}+\gamma_{k,l}=0\]
for all $l=1,\dots ,k-1$. Therefore, $\gamma_{k,l}=0$ for all $l=1,\dots ,k-1$.

For the remaining coefficients, by the following table 

\begin{center}
\begin{tabular}{|c|c|c|c|}
\hline
Entry & Information &Monomial & Its coefficient \\
\hline
$(1,3)$ &  & $m(g_{13}^{(2,j_1)})$ & $\pm \beta_{2,j_1}$ \\
\hline
$(1,3)$ &  & $u$ & $\pm 2\beta_{1,j_1}$ \\
\hline
\end{tabular}
\end{center}
where
\begin{center}
\begin{tabular}{ll}
$m(g_{13}^{(2,j_1)})=y_{12}^{j_1}z_{12}^2$ & \ \ \ \ \  $u=y_{12}^{j_1}y_{12}^k$, \\
\end{tabular}
\end{center}
we have $\beta_{2,j_1}=0$ and $\beta_{1,j_1}=0$, respectively.
\end{proof}
\end{proposition}

\subsection{Conclusion}

Since $\mathbb{F}$ is an infinite field and $B_{MN}\cap Id=B_{MN}\cap I$ for all $M,N$, we have the first main result of 
this paper.

\begin{theorem}\label{teoremaprincipal1}
Let $\mathbb{F}$ be an infinite field with char$(\mathbb{F})> 2$. If $*$ is an involution of the first kind on 
$UT_3(\mathbb{F})$ then $Id(UT_3(\mathbb{F}),*)$ is the T$(*)$-ideal generated
by the polynomials of Proposition \ref{ident}.
\end{theorem}

Note that this theorem is also true when char$(\mathbb{F})=0$. See \cite[Theorem 6.6]{vinkossca}.

\section{$\ast$-Central Polynomials for $UT_n(\mathbb{F})$}

Let $\mathbb{F}$ be a field (finite or infinite) of characteristic $\neq 2$. In this section we study the $\ast$-central polynomials for $UT_n(\mathbb{F})$, where $n\geq 3$ . 

Consider the involutions $*$ and $s$ in Section \ref{secaocomasduasinvolucoes}. If $\circ$ is an involution on $UT_n(\mathbb{F})$ then $\circ$ is equivalent either to $*$ or to $s$, see Section \ref{secaocomasduasinvolucoes}.
Thus 
\begin{equation} \label{igualdadepolinomioscentrais}
C(UT_n(\mathbb{F}),\circ)=C(UT_n(\mathbb{F}),*) \ \ \mbox{or} \ \  C(UT_n(\mathbb{F}),\circ)=C(UT_n(\mathbb{F}),s).
\end{equation}

\begin{theorem}
If $\circ$ is an involution on $UT_n(\mathbb{F})$ and $n\geq 3$ then 
\[C(UT_n(\mathbb{F}),\circ)=Id(UT_n(\mathbb{F}),\circ)+\mathbb{F}.\]
\end{theorem}

\begin{proof}
By (\ref{igualdadepolinomioscentrais}) we can suppose $\circ = *$ or $\circ = s$. In this case we have that
$e_{11}^{\circ}=e_{nn}.$  In particular, $A=e_{11}+e_{nn}$ and $B=e_{11}-e_{nn}$ are symmetric and skew-symmetric 
elements respectively.

Since 
\[C(UT_n(\mathbb{F}),\circ) \supseteq Id(UT_n(\mathbb{F}),\circ)+\mathbb{F}\]
we shall prove the inclusion $\subseteq$. Let $g(y_1,\ldots ,y_k,z_1,\ldots , z_s)\in C(UT_n(\mathbb{F}),\circ)$. 
Write 
\[g(y_1,\ldots ,y_k,z_1,\ldots , z_s)=f(y_1,\ldots ,y_k,z_1,\ldots , z_s)+\lambda\]
where $f(0,\ldots ,0,0,\ldots , 0)=0$ ($f$ without constant term) and $\lambda \in \mathbb{F}$. 

\

\noindent {\bf Claim 1:} $f(y_1,\ldots ,y_k,z_1,\ldots , z_s)$ is a polynomial identity for $\mathbb{F}$.

In fact, let $a_1,\ldots ,a_k,b_1,\ldots , b_s \in \mathbb{F}$. Write
\[f(a_1A,\ldots ,a_kA,b_1B,\ldots , b_sB)=\sum \alpha_{ij} e_{ij}.\]
Since $\alpha _{11}=f(a_1,\ldots ,a_k,b_1,\ldots , b_s)$, \ $\alpha_{22}=0$ and 
$f(y_1,\ldots ,y_k,z_1,\ldots , z_s) \in C(UT_n(\mathbb{F}),\circ)$ it follows that $\alpha_{11}=\alpha_{22}=0$ as desired.

\

 \noindent{\bf Claim 2:} $f(y_1,\ldots ,y_k,z_1,\ldots , z_s) \in Id(UT_n(\mathbb{F}),\circ)$.

Let   $A_1,\ldots ,A_k \in UT_n(\mathbb{F})^+$ and $B_1,\ldots ,B_s \in UT_n(\mathbb{F})^-$ where
\[A_l=\sum a_{ij}^le_{ij}  \ \ \mbox{and} \ \ B_l=\sum b_{ij}^le_{ij} .\]
Write
\[f(A_1,\ldots ,A_k,B_1,\ldots , B_s)=\sum \alpha_{ij} e_{ij}.\]
Since  
$f(y_1,\ldots ,y_k,z_1,\ldots , z_s) \in C(UT_n(\mathbb{F}),\circ)$ it follows that
\[f(A_1,\ldots ,A_k,B_1,\ldots , B_s)=\sum_{i=1}^n \alpha e_{ii},\]
where $\alpha = \alpha_{11}=\ldots =\alpha_{nn}$. Since 
$\alpha_{11}=f(a_{11}^1,\ldots ,a_{11}^k,b_{11}^1,\ldots , b_{11}^s)$, by Claim 1 we have $\alpha =0$ as desired.

By Claim 2 we have $g(y_1,\ldots ,y_k,z_1,\ldots , z_s)\in Id(UT_n(\mathbb{F}),\circ)+\mathbb{F}$. 
\end{proof}

\section*{Funding}
Dalton Couto Silva was supported by Ph.D. grant from Coordena\c{c}\~ao de Aperfei\c{c}oamento de Pessoal de N\'{i}vel Superior (CAPES). Dimas Jos\'e Gon\c{c}alves was partially supported by Funda\c{c}\~ao de Amparo \`a Pesquisa do Estado de S\~ao Paulo (FAPESP) grant No. 2018/23690-6.


\begin{thebibliography}{99}

\bibitem{Aljgiakar}
 Aljadeff, Eli; Giambruno, Antonio; Karasik, Yakov. Polynomial identities with involution, superinvolutions and the Grassmann envelope. Proc. Amer. Math. Soc. 145 (2017), no. 5, 1843--1857.


\bibitem{bahturinbook}
Bahturin, Yu. A. Identical relations in Lie algebras. Translated from the Russian by Bahturin. VNU Science Press, b.v., Utrecht, 1987. {\rm x}+309 pp.

\bibitem{brandaoplamen}
Brandão Jr., Antônio Pereira; Koshlukov, Plamen. Central polynomials for ${\Bbb Z}_2$-graded algebras and for algebras with involution. J. Pure Appl. Algebra 208 (2007), no. 3, 877--886.

\bibitem{colomboplamen}  
Colombo, Jones; Koshlukov, Plamen. Identities with involution for the matrix algebra of order two in characteristic $p$. Israel J. Math. 146 (2005), 337--355. 

\bibitem{drenskygiambruno}
Drensky, Vesselin; Giambruno, Antonio. Cocharacters, codimensions and Hilbert series of the polynomial identities for $2\times 2$ matrices with involution. Canad. J. Math. 46 (1994), no. 4, 718--733.



 \bibitem{gzbook}
 Giambruno, Antonio; Zaicev, Mikhail. Polynomial identities and asymptotic methods. Mathematical Surveys and Monographs, 122. American Mathematical Society, Providence, RI, 2005. xiv+352 pp. 

\bibitem{levchenkocarzero}
Levchenko, Diana V. Finite basis property of identities with involution of a second-order matrix algebra. (Russian) Serdica 8 (1982), no. 1, 42-56.

\bibitem{levchenkofinito}
Levchenko, Diana V. Bases of identities with involution of second-order matrix algebras over finite fields. (Russian) Serdica 10 (1984), no. 1, 55-67.


\bibitem{sviridova}
Sviridova, Irina.
Finite basis problem for identities with involution. (2014). arXiv: 1410.2233. 


\bibitem{ronalddimas}
Urure, Ronald Ismael Quispe; Gonçalves, Dimas José. Identities with involution for 2 $\times$ 2 upper triangular matrices algebra over a finite field. Linear Algebra Appl. 544 (2018), 223--253.


\bibitem{ronalddimascentral2}
Urure, Ronald Ismael Quispe; Gonçalves, Dimas José. Central polynomials with involution for the
algebra of $2 \times 2$ upper triangular matrices. Linear and Multilinear Algebra, DOI: 10.1080/03081087.2019.1648374.


\bibitem{vinkossca} Di Vincenzo, Onofrio Mario; Koshlukov, Plamen; La Scala, Roberto. Involutions for upper triangular matrix algebras. Adv. in Appl. Math. 37 (2006), no. 4, 541--568.

\end{thebibliography}
\end{document}